\documentclass[a4paper, 12pt, oneside]{amsart}
\usepackage{amssymb,amscd}

\usepackage{graphicx}

\usepackage[usenames,dvipsnames]{xcolor}

\usepackage[bibencoding=ascii, backend=bibtex, doi=false, url=false, isbn=false]{biblatex}
\bibliography{references}

\usepackage{eucal}

\usepackage[utf8]{inputenc}

\usepackage{bbm}

\newcommand{\E}{\mathbb{E}}
\newcommand{\V}{\mathbb{V}}

\newcommand{\N}{\mathbb{N}}
\newcommand{\Q}{\mathbb{Q}}
\newcommand{\R}{\mathbb{R}}

\newcommand{\Prim}{\mathbb{P}}

\newcommand{\xb}{{\bf x}}

\DeclareMathOperator{\diam}{diam}

\DeclareMathOperator{\pnt}{\raise 0.5mm \hbox{\large\bf.}}

\DeclareMathOperator{\Ent}{Ent}

\newtheoremstyle{thm}{}{}%
     {\em}
     {}
     {\bf}
     {.}
     {0.5em}
     {\thmname{#1}\thmnumber{ #2}\thmnote{ #3}}

\newtheoremstyle{def}{}{}%
     {\rm}
     {}
     {\bf}
     {.}
     {0.5em}
     {\thmname{#1}\thmnumber{ #2}\thmnote{ #3}}

\theoremstyle{thm}

\newtheorem{thm}{Theorem}[section]
\newtheorem{lem}[thm]{Lemma}
\newtheorem{cor}[thm]{Corollary}
\newtheorem{prop}[thm]{Proposition}

\theoremstyle{def}

\newtheorem{defi}[thm]{Definition}
\newtheorem{rem}[thm]{Remark}

\let\epsilon=\varepsilon

\newcommand{\X}{\mathbb{X}}
\newcommand{\OX}{{\bf N}}

\renewcommand{\P}{\mathbb{P}}
\newcommand{\ind}{\mathbbm{1}}

\newcommand{\x}{\textbf{x}}
\newcommand{\y}{\textbf{y}}

\begin{document}

\title[Concentration for Geometric Poisson Functionals]{Concentration bounds for  Geometric \\ Poisson Functionals: \\ Logarithmic Sobolev Inequalities Revisited}
\author{Sascha Bachmann}
\author{Giovanni Peccati}
\address{Sascha Bachmann \newline Institut f\"ur Mathematik \newline Universit\"at Osnabr\"uck \newline 49069 Osnabr\"uck, Germany \newline Email: sascha.bachmann@uos.de}
\address{Giovanni Peccati \newline Unit\'e de Recherche en Math\'ematiques \newline Facult\'e des Sciences, de la Technologie \newline et de la Communication, Universit\'e du Luxembourg \newline Email: giovanni.peccati@gmail.com \newline Webpage: https://sites.google.com/site/giovannipeccati/Home}

\begin{abstract}
We prove new concentration estimates for random variables that are functionals of a Poisson measure defined on a general measure space. Our results are specifically adapted to geometric applications, and are based on a pervasive use of a powerful logarithmic Sobolev inequality proved by L. Wu \cite{W_2000}, as well as on several variations of the so-called Herbst argument. We provide several applications, in particular to edge counting and more general length power functionals in random geometric graphs, as well as to the convex distance for random point measures recently introduced by M. Reitzner \cite{R_2013}.

\smallskip


\smallskip

\noindent{\bf 2010 AMS Classification:} 60D05; 60G57; 60C05
\end{abstract}

\maketitle

\section{Introduction}

\subsection{Overview}

Let $\eta$ be a Poisson random measure over some measurable space $(\X, \mathcal{X})$ such that $\mathcal{X}$ is countably generated, and assume that $\eta$ has a $\sigma$-finite intensity $\mu$. Let $F = F(\eta)$ be a real-valued functional of $\eta$ having finite expectation. In this paper, we are interested in proving several novel estimates for the upper and lower tails
$$
\P(F \geq \E F+  r) \quad \mbox{and} \quad \P(F \leq \E F  -r), \quad r>0,
$$
that are well adapted for geometric applications, with particular emphasis on quantities appearing in the modern theory of random geometric graphs -- see e.g. \cite{Penrose_book}.

\medskip

Our techniques are based on many variations of the so-called {\it Herbst argument} (see e.g. \cite{blm_book, BLM_2003, ledoux_book, M_2000}), basically consisting in using a logarithmic Sobolev inequality (or, alternatively, an integration by parts formula) in order to deduce a differential inequality involving the moment generating function $u\mapsto K(u) := \E[e^{uF}]$; solving the inequality then yields an upper bound on $K(u)$, implying in turn a tail estimate for $F$ by means of Markov's inequality. The main insight developed in the present paper is that, by carefully combining the {\it Mecke formula}  for Poisson point processes (see Section \ref{s:frame}) with logarithmic Sobolev inequalities such as the one in Theorem \ref{t:wu} below, one can deduce bounds on $K(u)$ involving quantities of a fundamental geometric nature. As discussed below, other approaches to concentration via the Herbst argument on the Poisson space (see e.g. \cite{BHP_2007, HP_2002, W_2000}) do not yield conditions that are amenable to geometric analysis. 

\subsection{Logarithmic Sobolev inequalities and a motivating example}

Our starting point is the following powerful Theorem \ref{t:wu}, proved by Wu in \cite{W_2000} (see also \cite{chafai}), and extending previous breakthrough findings contained in \cite{AL_2000, BL_1998}. Such a result involves two objects: (i) the {\it entropy} of a random variable $Z > 0$ with $\E Z<\infty$, that is defined as
\[
\Ent(Z) := \E(Z \log(Z)) - \E(Z) \log(\E Z ), 
\]
and (ii) the {\it difference} (or {\it add-one cost}) operator $DF$, that is defined for any $x \in\X$ as
\begin{align*}
D_xF(\eta) = F(\eta + \delta_x) - F(\eta),
\end{align*}
where $\delta_x$ denotes the Dirac mass at $x\in\X$.

\begin{thm}[(See Corollary 2.3 in \cite{W_2000})]\label{t:wu} For all $\lambda \in\R$ satisfying $\E (e^{\lambda F})<\infty$ we have
\begin{equation}\label{e:wu}
\Ent(e^{\lambda F}) \leq \E \left[ e^{\lambda F} \left(\int_{\X} \psi(\lambda D_x F(\eta)) \ d\mu(x)\right)\right],
\end{equation}
where $\psi(z) = ze^z - e^z + 1$.
\end{thm}

A typical way of applying \eqref{e:wu} to concentration estimates (via the Herbst argument) is demonstrated e.g. in \cite[Proposition 3.1]{W_2000}, where it is proved that, if $DF$ and $\int_\X (DF)^2 d\mu $ are almost surely bounded by positive constants $\beta$ and $c$, respectively, then the upper tail of $F$ is bounded by the function $$r\mapsto \exp\left[ - \frac{r}{2\beta}\log\left(1+\frac{\beta r}{c}\right) \right], \quad r>0.$$ Letting $\beta\to 0$, one deduces from this estimate that, if $DF\leq 0$ and $\int_\X (DF)^2 d\mu\leq c$, then,
$$
\P[F\geq \E F+ r] \leq \exp\left(- \frac{r^2}{2c}\right),\quad r>0,
$$
that is: such a concentration result only captures a Gaussian behaviour for the upper tail in the case of a non-increasing functional $F$ such that $\int_\X (DF)^2 d\mu$ is deterministically bounded. Apart from the monotonicity requirement on $F$, a crucial limitation of a result of this kind is that, in most examples where $F$ is a quantity arising in stochastic geometry (for instance, $F$ is an {\it edge-counting statistic} such as the ones considered in Section \ref{s:edge} below), the quantity $\int_\X (DF)^2 d\mu$ {\it does not admit any meaningful geometric interpretation} -- roughly because averaging $DF$ over the deterministic measure $\mu$ completely cancels the special role played by those points in $\X$ that belong to the support of $\eta$. One should contrast such a situation with the following statement, that will be proved later on as a special case of Corollary \ref{DevIneq1} (such a result also implies the already quoted Proposition 3.1 in \cite{W_2000}):

\begin{prop}\label{p:lh} Assume that there exists a finite constant $c>0$ such that, almost surely,
{
\begin{equation}
V^+ := \int_\X (D_xF)^2_- d\mu(x) + \int_{\X} (F(\eta) - F(\eta-\delta_x))^2_+ d\eta(x)\leq c,
\end{equation}
}
where $(u)_-$ and $(u)_+$ stand for the negative and positive part of $u\in \R$, respectively. Then, 
$$
\P(F\leq \E F +r) \leq \exp\left( - \frac{r^2}{2c}\right).
$$
\end{prop}
Proposition \ref{p:lh} is particularly interesting when $F$ is {\it non-decreasing}, {that is, when $DF\geq 0$}. Indeed, in this case one has that {$V^+ = \int_{\X} (F(\eta) - F(\eta-\delta_x))^2_+ d\eta(x)$} and, since the role of $\mu$ is now immaterial, the relation $V_+\leq c$ can be in principle verified by means of arguments of a purely geometric or combinatorial nature. For instance, we implement this strategy in Proposition \ref{convDistProp} below, where we use Proposition \ref{p:lh} in order to deduce a novel intrinsic proof of the {Gaussian} upper tail behaviour of the convex distance for point processes -- as recently introduced by Reitzner in \cite{R_2013}. 

As anticipated, the principal aim of this paper is to prove a large collection of statements with the same flavour as Proposition \ref{p:lh} (see Section 3), and then to apply them to random variables arising in the theory of random geometric graphs.

\subsection{Plan} Our work is organised as follows. After some preliminary facts discussed in Section 2, the subsequent Section 3 contains the statements of our main concentration estimates. Our results involve random variables having a form similar to the quantity $V^+$ introduced above, and largely generalise Proposition \ref{p:lh}. Proofs are detailed in Section 4. 

\smallskip
 
Section 5 presents several applications of the results of Section 3 to Poisson U-statistics of arbitrary order -- as defined in the seminal reference \cite{RS_2013} (see also \cite{BP_2012, DFRV, ET, LRP1, LRP2, PT_alea, ST_spa2012, LR_2015}). As a by-product of our analysis, in Proposition \ref{p:p} we also establish a new characterisation of square-integrable Poisson U-statistics.

\smallskip

The results of Section 5 are specialised in Section 6 to the case of edge-counting statistics associated with general random geometric graphs. Several  careful comparisons with the existing literature (in particular \cite{ERS_2015, RST_2013}) are presented.

\smallskip

Section 7 contains further estimates on U-statistics of order two, that are proved by adapting some techniques introduced in \cite{HRB, R_2003}. Geometric applications to edge-length functionals are discussed in detail.

\smallskip

As anticipated, in Section 8 we apply the estimates of Section 3, in order to deduce a novel intrinsic proof of the concentration estimates for the convex distance for random point measures established in \cite{R_2013}. Such a fundamental object generalises to the framework of random point processes the celebrated {\it convex distance} introduced by Talagrand in \cite{T_1995}; see \cite{RST_2013, LR_2015} for several applications. We stress that the problem of finding an intrinsic proof of the striking concentration results from \cite{R_2013} has been one of the main motivations for elaborating the theory developed in the present paper.

\subsection{Further remarks on the literature}

The inequalities obtained in this paper (as well as some techniques exploited in the proofs) are very close in spirit to those appearing in the seminal references \cite{blm_book, BLM_2003, M_2000}, where the so-called {\it entropy method} (roughly corresponding to a combination of the Herbst argument and of logarithmic Sobolev inequalities -- see e.g. \cite{ledoux_book}) is developed in the framework of functions of finite vectors of independent random elements. We recall that the results from \cite{blm_book, BLM_2003, M_2000} typically apply to random variables with the form $F = f(X_1,...,X_n)$, where $X = (X_1,...,X_n)$ is a vector of independent random elements and $f$ is some deterministic measurable function, and are based on a pervasive use of {\it random difference operators} of the type $$\Delta_if(X) = f(X) - f(X_1,...,X_{i-1}, X'_i, X_{i+1}, ...,X_n), \quad i=1,...,n,$$
where $X'$ is an independent copy of $X$. By inspection of the results presented below, it is not difficult to show that, in the case where the intensity $\mu$ of the Poisson measure $\eta$ is {\it finite} and {\it non-atomic}, some versions of the main results of the present paper could be obtained by implementing the following rough strategy:
\begin{enumerate}

\item Select a sequence of measurable partitions $\{B_1^{n},...,B_{k(n)}^{n} : n\geq 1\}$ of $\X$, in such a way that $k(n)\to\infty$ and $\max_{i=1,...,k(n)} \mu(B_i^{n})\to 0$, as $n\to \infty$.

\item Consider a random variable $F = F(\eta)$ and represent it in the form $F = f_n(X_{n,1},..,X_{n,k(n)})$, where $X_{n,i}$ is defined as the restriction of $\eta$ to the set $B_i^n$, and $f_n$ is some appropriate measurable mapping.

\item For a fixed $n$, prove a concentration estimate for $F(\eta)$ by applying the results from  \cite{blm_book, BLM_2003, M_2000} to $f_n(X_{n,1},..,X_{n,k(n)})$, in particular by considering an independent copy of $X_{n,1},..,X_{n,k(n)}$ defined in terms of an independent Poisson measure $\eta'$ on $\X$ with intensity $\mu$.

\item Let $n\to \infty$, and recover a bound involving quantities related to the add-one cost operator $DF$ described above, by exploiting the fact that independent Poisson measures with non-atomic intensities have almost surely disjoint supports.

\end{enumerate}

Apart from the fact that this approach only works with finite intensity measures without atoms, some investigations in this direction have convincingly shown us that (to the best of our expertise), in order for the step described at Point (iv) to take place in a meaningful way, one should systematically add to our statements some additional technical assumptions, that are indeed {\it not required} if one implements the direct approach based on the Mecke formula that is systematically adopted in this paper. An analogous phenomenon can be observed for instance in \cite[Theorem 4]{hr_2009}, where a weaker version of the Poincar\'e inequality on the Poisson space is deduced by means of a discretisation procedure similar to the one outlined above, and of the use of the classical Efron-Stein inequality. In view of these remarks, we decided not to directly exploit the connection with the entropy method on product spaces in the proofs of our main results. 

\medskip

Another collection of results that is relevant for our paper is contained in references \cite{BHP_2007, HP_2002}, where the authors obtain concentration estimates by applying integration by parts techniques, in particular by using the properties of the so-called {\it Ornstein-Uhlenbeck semigroup} associated with a given Poisson measure -- see also \cite{surg}. As in the already discussed examples from \cite{W_2000}, the estimates contained in these references have an equally problematic geometric interpretation, since they involve integrals of add-one cost operators with respect to the underlying intensity measure $\mu$. Moreover, in order to exploit some probabilistic representation of the Ornstein-Uhlenbeck semigroup, one has also to work on {\it extended} probability spaces. It is a natural question to ask whether the Mecke formula could be combined with some of the estimates from  \cite{BHP_2007, HP_2002} in order to obtain concentration inequalities that are adapted to a geometric framework. We prefer to think of this issue as a separate problem, and leave it open for further research.

\subsection{Acknowledgments} The authors wish to thank M. Reitzner and Ch.$\!$ Th\"ale for useful discussions. S. Bachmann is partially supported by the German Research Foundation DFG-GRK 1916. G. Peccati is partially supported by the grant F1R-MTH-PUL-12PAMP (PAMPAS) at Luxembourg University.

\section{Framework}\label{s:frame}

For the rest of the paper, we shall denote by $(\X, \mathcal{X}, \mu)$ a $\sigma$-finite measure space, such that the $\sigma$-field $\mathcal{X}$ is countably generated and $\mu(\mathbb{X})>0$. We write $\eta$ to indicate a Poisson point process on $(\X, \mathcal{X})$. This means that $\eta = \{\eta(A) : A\in \mathcal{X}_0\}$ is a collection of random variables, defined on some probability space $(B, \mathcal{B},\Prim)$ and indexed by the elements of $\mathcal{X}_0 = \{A\in \mathcal{X} : \mu(A)<\infty\}$, such that the following properties are satisfied: {\bf (i)} for every fixed $A\in \mathcal{X}_0$, $\eta(A)$ is a Poisson random variable with parameter $\mu(A)$, and {\bf (ii)} for every collection of pairwise disjoint $A_1,...,A_n\in \mathcal{X}_0$, one has that the random variables $\eta(A_1),...,\eta(A_n)$ are stochastically independent. 

\medskip

As usual, we interpret $\eta$ as a random element in the space
$\OX = {\bf N}(\mathbb{X})$ of integer-valued $\sigma$-finite measures $\xi$
on $\X$ equipped with the smallest $\sigma$-field $\mathcal{N}  $
making the mappings $\xi\mapsto\xi(B)$ measurable for all
$B\in\mathcal{X}$; see e.g. \cite{SW_2008} or \cite{LP_2011}. The standard notation $x\in \eta$ is shorthand in order to indicate that the point $x$ is an element of the support of $\eta$. We write $\hat{\eta}$ for the compensated random (signed) measure $\eta-\mu$. We shall write $F = F(\eta)$ to indicate that a given random variable $F$ can be written in the form $F=\mathfrak{f}(\eta)$, $\mathbb{P}$-a.s., for some measurable function
$\mathfrak{f}: \OX \rightarrow\R$; such a function $\mathfrak{f}$ (which is uniquely determined by $F$ up to sets of $\mathbb{P}$-measure zero) is customarily called a {\it representative} of $F$, and $F$ is called a {\it Poisson functional}.

\begin{rem}[(Some conventions)]\label{r:talich} In what follows, we will indifferently use the notation $F$ and $F(\eta)$ if there is no ambiguity. Also, for any $\xi\in\OX$, the notation $F(\xi)$ refers to $\mathfrak{f}(\xi)$, where $\mathfrak{f}$ is a fixed representative of $F$. Finally, we observe that, in the statements of some of our main results, we will often work under the assumption that the add-one cost operator $D_xF(\xi)$ verifies a given property $\mathcal{P}$ (for instance, $D_xF(\xi)\leq 0$) for every $x\in \X$ and every $\xi \in {\bf{N}}$: this requirement means of course that there exists a representative $\mathfrak{f}$ of $F$ such that the quantity $\mathfrak{f}(\xi + \delta_x) - \mathfrak{f}(\xi)$ verifies $\mathcal{P}$ for every $x\in \X$ and every $\xi \in {\bf{N}}$. 
\end{rem}
\medskip

We will systematically use the following standard notation: for every $\xi \in \OX$,
\begin{align}\label{xiNeq}
\xi_{\neq}^k := \{{\bf x} = (x_1,...,x_k) : x_i \in \xi, \, \forall i=1,...,k, \, \text{and}\,  x_i\neq x_j, \, i\neq j\}.
\end{align}

A result that we shall use in several occasions (and that in some sense represents the backbone of our approach) is the following well-known {\it Slivnyak-Mecke formula}: for every $m\geq 1$ and every non-negative measurable function $H$ on $\OX\times \mathbb{X}^m$, one has that 
{
\begin{eqnarray}\label{e:smecke}
&&\E\left[\int_{\X^m} H(\eta, x_1,...,x_m) d\eta^{(m)}(x_1,\ldots,x_m) \right] \\ \notag &&= \int_{\mathbb{X}}\cdots \int_{\mathbb{X}}  \E\left[H\left(\eta+\sum_{i=1}^m \delta_{x_i},x_1, \ldots , x_m\right)\right] d\mu(x_1)\cdots d\mu(x_m),
\end{eqnarray}
where $\delta_x$ stands for the Dirac mass at $x$ and $\eta^ {(m)}$ is the point process on $\X^m$ with support $\eta_{\neq}^m$ and $\eta^{(m)}(\{\xb\}) = \eta(\{x_1\})\cdots\eta(\{x_m\})$ for any $\xb=(x_1,\ldots,x_m)\in\eta_{\neq}^m$}. A standard proof of the fundamental relation \eqref{e:smecke} can be found e.g. in \cite[Theorem 3.2.5 and Corollary 3.2.3]{SW_2008}, in the case of a non-atomic intensity $\mu$. The result extends straightforwardly to the case of a general $\sigma$-finite measure $\mu$ -- see e.g. \cite{LP_2011}. When specialised to the case $m=1$, relation \eqref{e:smecke} is known as {\it Mecke formula}, and boils down to the following identity: for every non-negative measurable function $H$ on $\OX\times \mathbb{X}$, one has that 
{
\begin{equation}\label{e:mecke}
\E\left[\int_\X H(\eta, x) d\eta(x)\right] = \int_{\mathbb{X}} \E[H(\eta+\delta_x,x)] d\mu(x).
\end{equation}
 }
\smallskip

For the rest of the paper, for every integer $k\geq 1$ and every real $p>0$, we will write $L^p(\mu^k) := L^p(\mathbb{X}^k , \mathcal{X}^{\otimes k}, \mu^k)$, and also use the shorthand notation $L^p(\mu^1) = L^p(\mu)$. In Section \ref{ss:zut}, the symbol $L^2(\mu^0)$ is used to denote the real line $\R$, endowed with the usual Euclidean inner product.

\section{Deviation Inequalities for Poisson Functionals} \label{DIPF}
In the following, we are going to develop new tools for proving deviation inequalities of Poisson functionals. Our approach is an adaptation of the entropy method for product space functionals that was particularly investigated in \cite{BLM_2003}.
\medskip

\medskip

The heart of the method we are about to present is the modified logarithmic Sobolev inequality stated below. {For the remainder of the section, we consider a Poisson functional $F$.}
As above, we will use the difference (or add-one cost) operator $DF$, that is defined for any $(x,\xi)\in\X\times\OX$ by
\begin{align*}
D_xF(\xi) = F(\xi + \delta_x) - F(\xi),
\end{align*}
where $\delta_x$ denotes the Dirac mass at $x\in\X$. To shorten notations, we write
\begin{align*}
D_x^IF(\xi) = D_xF(\xi) \ind\{(x,\xi)\in I\}
\end{align*}
whenever $I\subseteq \X \times \OX$ is measurable, $x\in\X$ and $\xi \in\OX$. In the same spirit we will also use the notations
\begin{align*}
D_x^{\geq\beta}F(\xi) &= D_xF(\xi) \ind\{D_xF(\xi)\geq \beta\},\\
D_x^+F(\xi) &= D_xF(\xi) \ind\{D_xF(\xi)\geq 0\}.
\end{align*}
The quantities $D_x^{\leq\beta}F(\xi)$ and $D_x^-F(\xi)$ are defined analogously, and so are the operators $D_x^{>\beta}F$ and $D_x^{<\beta}F$ (with strict inequalities). The following observation is derived by combining Wu's modified logarithmic Sobolev inequality for Poisson point processes \eqref{e:wu} with the Mecke formula \eqref{e:mecke}.

\begin{prop} \label{EntIneq}
Let $I\subseteq \X\times\OX$ be a measurable set. Then for all $\lambda \in\R$ satisfying $\E (e^{\lambda F})<\infty$ we have
\[
\Ent(e^{\lambda F}) \leq \E \left[ e^{\lambda F} \left(\int_{\X} \psi(\lambda D^I_x F(\eta)) \ d\mu(x) + \int_\X \phi(-\lambda D^{I^c}_x F(\eta-\delta_x)) \ d\eta(x)\right)\right],
\]
where $\phi(z) = e^z - z - 1$ and $\psi(z) = ze^z - e^z + 1$.
\end{prop}

For any $\beta\in\R$ we define the random variables $V_\beta^+ = V_\beta^+(F)$ and $V_\beta^- = V_\beta^-(F)$ by
\begin{align*}
V_\beta^+ &= \int_{\X} (D^{\leq \beta}_x F(\eta))^2 \ d\mu(x) + \int_\X (D^{ > \beta}_x F(\eta-\delta_x))^2  d\eta(x),\\
V_\beta^- &= \int_{\X} (D^{\geq \beta}_x F(\eta))^2 \ d\mu(x) + \int_\X (D^{ < \beta}_x F(\eta-\delta_x))^2  d\eta(x).
\end{align*}
Note that we will write $V^+ = V_0^+$ and $V^- = V_0^-$. The notation is in correspondence with \cite{BLM_2003} where the entropy method for product spaces was investigated. The upcoming result can be regarded as a generalized analogue of \cite[Theorem 2]{BLM_2003} for the Poisson space. The proof is similar to the product space version where Proposition \ref{EntIneq} takes now the role of the log Sobolev inequality. The generalization is achieved using arguments similar to those in the proof of \cite[Proposition 3.1]{W_2000}.

To get prepared for the presentation of the theorem, for $\beta\in\R$ and $z>0$, let
\begin{align*}
\Phi_\beta(z) = \begin{cases}
\psi(z\beta)/(z\beta^2)& \beta>0\\ 
z/2& \beta=0\\
\phi(-z\beta)/(z\beta^2) & \beta<0,
\end{cases}
\end{align*}
and
\begin{align*}
\Psi_\beta(z) = \begin{cases}
\phi(z\beta)/(z\beta^2)& \beta>0\\ 
z/2& \beta=0\\
\phi(-z\beta)/(z\beta^2) & \beta<0,
\end{cases}
\end{align*}
where $\phi$ and $\psi$ are as in Proposition \ref{EntIneq}. Note that we will frequently use the fact that these functions are non-decreasing.

\begin{thm} \label{EntIneqs}
{Assume that the Poisson functional $F$ is integrable}. Let $\lambda>0$ be such that $\E \exp(\lambda F)<\infty$. Then for any $\theta>0$ satisfying $\Phi_\beta(\lambda)\theta < 1$, we have
\begin{align}\label{EntIneq1}
\log \E[\exp(\lambda(F-\E F))] &\leq \frac{\Psi_\beta(\lambda) \theta}{1-\Phi_\beta(\lambda)\theta} \log \E \left[ \exp\left(\frac{\lambda V_\beta^+}{\theta}\right)\right].
\end{align}
Let $\lambda>0$ be such that $\E \exp(-\lambda F)<\infty$. Then for any $\theta>0$ with $\Phi_{-\beta}(\lambda)\theta < 1$, we have
\begin{align}\label{EntIneq2}
\log \E[\exp(-\lambda(F-\E F))] &\leq \frac{\Psi_{-\beta}(\lambda) \theta}{1-\Phi_{-\beta}(\lambda)\theta} \log \E \left[\exp\left(\frac{\lambda V_\beta^-}{\theta}\right)\right].
\end{align}
Assume that $F$ is not necessarily integrable and that one of the following conditions is satisfied:
\begin{enumerate}
\item $\beta>0$ and $D_x F(\xi) \leq \beta$ holds for all $(x,\xi)\in\X\times\OX$,
\item $\beta<0$ and $D_x F(\xi) \geq \beta$ holds for all $(x,\xi)\in\X\times\OX$,
\item $\beta = 0$.
\end{enumerate}
Then, the relation $\E \exp(\lambda V_\beta^+ / \theta)<\infty$ implies that $\E |F|<\infty$ and $\E \exp(\lambda F)<\infty$. Also, the relation $\E \exp(\lambda V_\beta^- / \theta)<\infty$ implies that $\E |F|<\infty$ and $\E \exp(-\lambda F)<\infty$.
\end{thm}

With the above results the methods for deriving deviation inequalities presented in \cite{BLM_2003} naturally carry over to the Poisson space. In the following we present some variations of these techniques that will be used for the applications later on.

In the case when $V_\beta^+$ and $V_\beta^-$ for $\beta\neq 0$ are almost surely bounded by a constant, the above entropy inequalities yield exponential tails for the random variable $F(\eta)$. If $V^+$ and $V^-$ are almost surely bounded (i.e. $\beta = 0$), we even obtain Gaussian tails.  Note that Wu's deviation inequality \cite[Proposition 3.1]{W_2000} is implied by {the following} more general results.

\begin{cor} \label{DevIneq1}
Assume that $F$ satisfies $V_\beta^+ \leq c$ almost surely. Then $F$ is integrable and the following statements hold:
\begin{enumerate}
{
\item If either condition (i) or (ii) of Theorem \ref{EntIneqs} is satisfied, then for all $r\geq 0$,
\begin{align*}
\P(F \geq\E F + r) &\leq \exp\left(-\left(\frac{c}{\beta^2} + \frac{r}{|\beta|}\right)\log\left(1+\frac{|\beta| r}{c}\right) + \frac{r}{|\beta|}\right)\\
&\leq \exp \left(-\frac{r}{2|\beta|} \log\left(1 + \frac{|\beta| r}{c}\right)\right).
\end{align*}
}
\item If $\beta = 0$, that is if $V^+\leq c$ holds almost surely, then for all $r\geq 0$,
\begin{align*}
\P(F\geq\E F + r) \leq \exp\left(-\frac{r^2}{2c}\right).
\end{align*}
\end{enumerate} 
\end{cor}

We continue with the corresponding version for the lower tail. This corollary is obtained in the same way as the above one where inequality (\ref{EntIneq2}) is used instead of (\ref{EntIneq1}). The proof is therefore omitted.

\begin{cor} \label{DevIneq12}
Assume that $F$ satisfies $V_\beta^- \leq c$ almost surely. Then $F$ is integrable and the following statements hold:
\begin{enumerate}
{
\item If either condition (i) or (ii) of Theorem \ref{EntIneqs} is satisfied, then for all $r\geq 0$,
\begin{align*}
\P(F \leq\E F - r) &\leq \exp\left(-\left(\frac{c}{\beta^2} + \frac{r}{|\beta|}\right)\log\left(1+\frac{|\beta| r}{c}\right) + \frac{r}{|\beta|}\right)\\
&\leq \exp \left(-\frac{r}{2|\beta|} \log\left(1 + \frac{|\beta| r}{c}\right)\right).
\end{align*}
}
\item If $\beta = 0$, that is if $V^-\leq c$ holds almost surely, then for all $r\geq 0$,
\begin{align*}
\P(F\leq\E F - r) \leq \exp\left(-\frac{r^2}{2c}\right).
\end{align*}
\end{enumerate} 
\end{cor}

The following result is useful to obtain deviation inequalities under less restrictive boundedness conditions on $V^+$.

\begin{cor} \label{cor2}
Assume that $F\geq 0$ and that there is a {random variable $G\geq 0$} and an $\alpha\in[0,2)$ such that almost surely
\begin{align*}
V^+ \leq G F^\alpha.
\end{align*}
Let $\theta >0$ and $\lambda\in (0,2/\theta)$ be such that $\E\exp(\lambda G/\theta) < \infty$. Then $\E F^{1-\alpha/2} <\infty$ and
\begin{align*}
\log \E(\exp(\lambda(F^{1-\alpha/2}-\E F^{1-\alpha/2}))) &\leq \frac{\lambda \theta}{2-\lambda\theta} \log \E \left( \exp\left(\frac{\lambda G}{\theta}\right)\right).
\end{align*}
\end{cor}

In the case when the random variable $G$ in the above corollary is just a constant, we obtain the following deviation inequality for the upper tail.

\begin{cor} \label{DevIneq2}
{Assume that $F\geq 0$} and that for some $\alpha\in[0,2)$ and $c>0$ we have almost surely
\begin{align*}
V^+ \leq c F^\alpha.
\end{align*}
Then $F$ is integrable and for all $r\geq 0$,
\begin{align*}
\P(F\geq\E F + r) \leq \exp\left(-\frac{((r+\E F)^{1-\alpha/2} - (\E F)^{1-\alpha/2})^2}{2c}\right).
\end{align*}
\end{cor}

The next result is a variation of Corollary \ref{cor2} for Poisson functionals that are not necessarily non-negative. This is the Poisson space analogue of \cite[Theorem 5]{BLM_2003}. 

\begin{thm}\label{arbSignThm}
Assume that the Poisson functional $F$ is integrable and that for some $a>0$ and $b \geq 0$ we have almost surely
\begin{align*}
V^+ \leq a F + b.
\end{align*}
Then for any $\lambda\in (0,2/a)$ we have $\E(e^{\lambda F})<\infty$ and
\begin{align*}
\log \E (\exp( \lambda(F-\E F))) \leq \frac{\lambda^2}{2-a\lambda} (a\E F+b).
\end{align*}
Moreover, for any $r\geq 0$,
\begin{align*}
\P(F\geq \E F + r) \leq \exp\left(-\frac{r^2}{2a\E F + 2b + ar/3}\right).
\end{align*}
\end{thm}

{We continue with a result that applies whenever $F$ is non-decreasing and $V^-$ is non-decreasing and integrable. In this case, the random variable $F$ has a Gaussian lower tail:}

\begin{thm} \label{LowerDevIneq}
Assume that the Poisson functional $F$ satisfies
\begin{align*}
{ D_x F(\xi), \ D_x V^-(\xi) \geq 0 \ \ \text{for all} \ \ (x,\xi)\in\X\times\OX} \ \ \text{and} \ \ \E V^-(\eta) < \infty.
\end{align*}
Then $F = F(\eta)$ is integrable and for all $r\geq 0$ we have
\begin{align*}
\P(F\leq \E F -r) \leq \exp\left(-\frac{r^2}{2 \E V^-}\right).
\end{align*}
\end{thm}

\begin{rem} \label{incVRem}
{Assume that the Poisson functional $F$ is non-decreasing.} A sufficient condition for the assumption $DV^-\geq 0$ in the above theorem is that the second interation of the difference operator of $F$ is non-negative. Indeed, assume that
\begin{align*}
D_zD_x F(\xi) \geq 0 \ \ \text{for all} \ \ (z,x,\xi)\in\X\times\X\times\OX.
\end{align*}
Then, $D_xF(\xi+\delta_z) \geq D_xF(\xi)\geq 0$ and hence $D_xF(\xi+\delta_z)^2 \geq D_xF(\xi)^2$ for all $(z,x,\xi)\in\X\times\X\times\OX$. So we see that also $D(DF)^2$ is non-negative, thus yielding
\begin{align*}
D V^- = D \int_\X (DF)^2d\mu = \int_\X D(DF)^2d\mu\geq 0.
\end{align*}
\end{rem}
\bigskip

{We conclude this section with a result that deals with the situation when $F$ is non-decreasing and the difference operator $DF$ is bounded. In this case}, we obtain a deviation inequality for the lower tail by controlling the random variable $V^+$. This is a Poisson space analogue of \cite[Theorem 13]{M_2006}.

\begin{thm}\label{vPlusLowerThm}
{Assume that $F\geq 0$ and }that for some $a> 0$ we have
\begin{align*}
{ 0\leq D_x F(\xi)\leq 1 \ \ \text{for any} \ \ (x,\xi)\in\X\times\OX} \ \ \text{and almost surely} \ \ V^+(\eta) \leq aF(\eta).
\end{align*}
Then $F$ is integrable and for any $r\geq 0$ we have
\begin{align*}
\P(F\leq \E F - r) \leq \exp\left(-\frac{r^2}{2 \max(a,1) \E F}\right).
\end{align*} 
\end{thm}

\section{Proofs}

We begin with the proof of the crucial logarithmic Sobolev type inequality, namely Proposition \ref{EntIneq}, that is the foundation of our techniques.

\begin{proof}[Proof of Proposition \ref{EntIneq}]
By \eqref{e:wu}, the inequality
\begin{align*}
\Ent(e^{\lambda F}) &\leq \E \left[ e^{\lambda F}\int_{\X}\psi(\lambda D_x F) \ d\mu(x)\right]
\end{align*}
holds. Now, since $\psi(0) = 0$, $\phi(0)=0$ and $\psi(z) = e^z \phi(-z)$ for any $z\in\R$, we have
\[
\psi(\lambda D_xF(\eta)) = \psi(\lambda D^I_xF(\eta)) + e^{\lambda D_xF(\eta)}\phi(-\lambda D^{I^c}_xF(\eta)).
\]
Hence, we compute
\begin{align*}
\Ent(e^{\lambda F})  &\leq \E \ \int_{\X} e^{\lambda F}\psi(\lambda D^I_x F) \ d\mu(x) + \E \int_{\X} e^{\lambda D_x F + \lambda F} \phi(-\lambda D^{I^c}_x F) \ d\mu(x)\\
&= \E \ \int_{\X} e^{\lambda F}\psi(\lambda D^I_x F) \ d\mu(x) + \E \int_{\X} e^{\lambda F(\eta + \delta_x)} \phi(-\lambda D^{I^c}_x F) \ d\mu(x)\\
&= \E \ \int_{\X} e^{\lambda F}\psi(\lambda D^I_x F) \ d\mu(x) + \E \int_\X e^{\lambda F} \phi(-\lambda D^{I^c}_x F(\eta -\delta_x))  d\eta(x).
\end{align*}
At this, the last equality holds by the Mecke formula \eqref{e:mecke}.
\end{proof}

The following lemma will be used occasionally in the upcoming proofs.

\begin{lem} \label{ExpLem}
Assume that for some $\beta\in\R$ we have $\E V_\beta^+ < \infty$ or $\E V_\beta^- < \infty$. Then $F$ is integrable.
\end{lem}

\begin{proof}
The proof uses a truncation argument that is standard in this context, see e.g. the proof of \cite[Proposition 3.1]{W_2000}. The statement for $V_\beta^-$ is proved in the same way than for $V_\beta^+$. Consider for any $n\in\N$ the truncation
\[
F_n = \min(\max(F,-n),n).
\]
Then $\E (F_n)^2<\infty$, hence the Poincar\'e inequality for Poisson point processes (see e.g. \cite[Remark 1.4]{W_2000}) together with the Mecke formula \eqref{e:mecke} yield
\begin{align*}
\V F_n \leq \E \int (D_x F_n(\eta))^2 d \mu(x) \leq \E \int (D_x F(\eta))^2d \mu(x) = \E V_\beta^+.
\end{align*}
Therefore, $\sup_{n\in\N}\V F_n \leq \E V_\beta^+ < \infty$. Now, if there is a subsequence $\{F_{n_k}\}_{k\in\N}$ satisfying $\lim_{k\to\infty}\E F_{n_k} = \pm\infty$, then $F_{n_k}\to \pm\infty$ in probability. This would be a contradiction to $F_n \to F$ in probability. We see that the family $\{\E F_n\}_{n\in\N}$ is bounded. Together with $\sup_{n\in\N}\V F_n < \infty$ this implies also $ \sup_{n\in\N}\E (F^2_n) < \infty$. Thus, the family $\{F_n\}_{n\in\N}$ is uniformly integrable. In particular, as desired, we have $\E |F| < \infty$.
\end{proof}

We continue with the proof of Theorem \ref{EntIneqs}. As in the proof of the product space version \cite[Theorem 2]{BLM_2003}, we also need \cite[Lemma 11]{M_2000}. This result states that for any $\lambda>0$ and any two random variables $X$ and $Y$ satisfying $\E (e^{\lambda X}), \E (e^{\lambda Y}) <\infty$, we have
\begin{align} \label{decoupling}
\frac{\lambda \E(X e^ {\lambda Y})}{\E(e^ {\lambda Y})} \leq \frac{\lambda \E(Y e^ {\lambda Y})}{\E(e^ {\lambda Y})} + \log \E(e^{\lambda X}) - \log \E(e^{\lambda Y}).
\end{align}
\begin{proof}[Proof of Theorem \ref{EntIneqs}]
We prove (\ref{EntIneq1}). We only deal with the case $\beta\neq 0$, whereas the case $\beta = 0$ can be obtained with by similar arguments. To prove the desired inequality we adapt the proof of \cite[Theorem 2]{BLM_2003} and combine this with arguments from the proof of \cite[Proposition 3.1]{W_2000}. Let $\phi$ and $\psi$ be as in Proposition \ref{EntIneq}. Then $\psi(z)/z^2$ and $\phi(z)/z^2$ are non-decreasing. Hence, for any $u\in(0,\lambda]$ we have
\begin{align*}
\psi(u z) & \leq \rlap{$(\psi(u\beta)/\beta^2)z^2$}\phantom{(\phi(-u\beta)/\beta^2)z^2} \leq u \Phi_\beta(u) z^2 \ \ \text{for} \ z \leq \beta,\\
\phi(-u z) &\leq (\phi(-u\beta)/\beta^2)z^2 \leq u \Phi_\beta(u) z^2 \ \ \text{for} \ z \geq \beta.
\end{align*}
Together with $\psi(0)=\phi(0)=0$ this gives
\begin{align*}
\psi(u D^{\leq\beta}_x F(\eta)) &\leq u \Phi_\beta(u) (D^{\leq\beta}_x F(\eta))^2,\\
\phi(-u D^{ >\beta}_x F(\eta-\delta_x)) &\leq u \Phi_\beta(u) (D^{ >\beta}_x F(\eta-\delta_x))^2.
\end{align*}
Hence, taking $I=\{(x,\xi) \in \X\times\OX : D_xF(\xi) \leq \beta\}$, it follows from Proposition \ref{EntIneq} that
\begin{align*}
\Ent(e^{u F}) &\leq \E \left[e^{u F} \left(\int_{\X} \psi(u D^{\leq\beta}_x F(\eta)) \ d\mu(x) + \int_\X \phi(-u D^{>\beta}_x F(\eta-\delta_x)) d\eta(x)\right)\right]\\
&\leq u \Phi_\beta(u) \ \E \left[ e^{u F} \left(\int_{\X} (D^{\leq \beta}_x F(\eta))^2 \ d\mu(x) + \int_\X (D^{ >\beta}_x F(\eta-\delta_x))^2  d\eta(x)\right)\right]\\
&= u\Phi_\beta(u) \ \E(V_\beta^+e^{u F}).
\end{align*}
Moreover, taking $X  =V_\beta^+/\theta$ and $Y = F$, it follows from (\ref{decoupling}) that
\begin{align*}
\frac{\E(V_\beta^+ e^ {u F})}{\E(e^ {u F})} \leq \frac{\theta \E(F e^ {u F})}{\E(e^ {u F})} + \frac{\theta}{u}\log \E(e^{u V_\beta^+ / \theta}) - \frac{\theta}{u}\log \E(e^{u F}).
\end{align*}
{Invoking the definition of the entropy, it follows from the last two displays that}
\begin{align*}
\frac{\E(F e^{u F})}{u \E(e^{u F})} - \frac{\log \E (e^{u F})}{u^2} 
\leq 
\Phi_\beta(u)\left(\frac{\theta \E(F e^ {u F})}{u\E(e^ {u F})} + \frac{\theta}{u^2}\log \E(e^{u V_\beta^+ / \theta}) - \frac{\theta}{u^2}\log \E(e^{u F})\right).
\end{align*}
Since by assumption $\Phi_\beta(u)\theta\leq \Phi_\beta(\lambda)\theta < 1$, the latter inequality is equivalent to
\begin{align*}
\frac{\E(F e^{u F})}{u \E(e^{u F})} - \frac{\log \E (e^{u F})}{u^2} 
\leq 
\frac{\Phi_\beta(u)\theta \log \E(e^{u V_\beta^+ / \theta})}{u^2 (1-\Phi_\beta(u) \theta)}.
\end{align*}
Defining  $h(u) = \frac{1}{u} \log \E(e^{u F})$ and $g(u) = \log\E(e^{u V_\beta^+})$, the above estimate can be restated as follows:
\begin{align*}
h'(u) \leq \frac{\Phi_\beta(u)\theta g(u/\theta)}{u^2(1-\Phi_\beta(u)\theta)}
\end{align*}
for any $u\in(0,\lambda]$. Since $\lim_{u\to 0+} h(u) = \E F$, integration from $0$ to $\lambda$ gives
\begin{align}\label{diffineq}
h(\lambda) \leq \E F + \int_0^\lambda \frac{\Phi_\beta(u)\theta g(u/\theta)}{u^2(1-\Phi_\beta(u)\theta)} \ du.
\end{align}
It's a well known fact that the logarithm of a moment generating function is convex, hence $g$ is convex on the interval $[0,\lambda/\theta]$. In particular, we have for any $u\in (0,\lambda]$ that
\begin{align*}
\frac{g(u/\theta)}{u(1-\Phi_\beta(u)\theta)} \leq \frac{\tfrac{g(\lambda/\theta)}{\lambda} \ u}{u(1-\Phi_\beta(u)\theta)} \leq \frac{g(\lambda/\theta)}{\lambda(1-\Phi_\beta(\lambda) \theta)}.
\end{align*}
Hence,
\begin{align*}
\int_0^\lambda \frac{\Phi_\beta(u)\theta g(u/\theta)}{u^2(1-\Phi_\beta(u)\theta)} \ du \leq \frac{\theta g(\lambda/\theta)}{\lambda(1-\Phi_\beta(\lambda) \theta)} \int_0^\lambda \frac{\Phi_\beta(u)}{u} \ du.
\end{align*}
In the case $\beta < 0$, we bound the integral on the right hand side by $\Phi_\beta(\lambda) = \Psi_\beta(\lambda)$. This works since $\Phi_\beta(u)/u$ is non-decreasing. In the case $\beta > 0$, the integral can be explicitly computed {and one obtains $\int_0^\lambda (\Phi_\beta(u)/u) du = \phi(\lambda\beta)/(\lambda\beta^ 2) = \Psi_\beta(\lambda)$}. Combining this with (\ref{diffineq}) gives
\begin{align*}
\log \E(e^{\lambda F}) \leq \lambda\E F + \frac{\Psi_\beta(\lambda) \theta }{1-\Phi_\beta(\lambda) \theta} \ g(\lambda/\theta).
\end{align*}
This proves inequality (\ref{EntIneq1}). Repeating the above reasoning for $-F$ instead of $F$  where the set $I$ is replaced by its complement proves inequality (\ref{EntIneq2}).

To prove the second part of the theorem, assume that one of the conditions (i) to (iii) is satisfied. For $n\in\N$ consider the truncated random variables
\[
F_n = \min(\max(F,-n),n).
\]
We will now conclude that if $\E\exp(\lambda V_\beta^ + / \theta) < \infty$, then $F$ is integrable and the family of random variables
\begin{align}\label{famRV}
\{\exp(\lambda(F_n-\E F_n))\}_{n\in\N}
\end{align}
converges in probability to $\exp(\lambda(F-\E F))$ and is uniformly integrable. Thus, it follows that $\E \exp(\lambda(F-\E F)) < \infty$ and hence also $\E\exp(\lambda(F))<\infty$.

Integrability of $F$ follows from Lemma \ref{ExpLem} since the assumption $\E(\exp(\lambda V_\beta^+/\theta)) < \infty$ implies that $\E V_\beta^+ < \infty$.
By {dominated} convergence, integrability of $F$ now implies the convergence in probability of the sequence in (\ref{famRV}).

To prove the uniform integrability, first observe, that if (i), (ii) or (iii) holds, then
\begin{align*}
V_\beta^+(F_n)\leq V_\beta^+(F) = V_\beta^+.
\end{align*}
Also note that we can choose $\nu>1$ such that $\Phi_\beta(\nu\lambda)\nu\theta < 1$. Then $\E(e^{\lambda\nu F_n})<\infty$ for all $n\in\N$, so it follows from (\ref{EntIneq1}) that
\begin{align*}
\log \E(\exp(\lambda\nu(F_n-\E F_n))) &\leq \frac{\Psi_\beta(\nu\lambda)\nu \theta}{1-\Phi_\beta(\nu\lambda)\nu \theta} \log \E \left( \exp\left(\frac{\lambda V_\beta^+(F_n)}{\theta}\right)\right)\\
&\leq \frac{\Psi_\beta(\nu\lambda)\nu \theta}{1-\Phi_\beta(\nu\lambda)\nu \theta} \log \E \left( \exp\left(\frac{\lambda V_\beta^+}{\theta}\right)\right) < \infty.
\end{align*}
Denoting the map $x\mapsto x^\nu$ by $\Lambda$, the above inequality yields
\begin{align*}
\sup_{n\in\N}\E[\Lambda(\exp(\lambda(F_n-\E F_n)))] < \infty.
\end{align*}
By the Theorem of de la Vall\'ee-Poussin this implies uniform integrability of the family in (\ref{famRV}).

Repeating the above reasoning for $-F$ instead of $F$  where inequality (\ref{EntIneq2}) is used instead of (\ref{EntIneq1}) proves the corresponding statement for $V_\beta^-$.
\end{proof}

{
\begin{proof}[Proof of Corollary \ref{DevIneq1}]
It follows from Lemma \ref{ExpLem} that $F$ is integrable. Theorem \ref{EntIneqs} yields
\[
\log \E (e^{\lambda (F - \E F)}) \leq \inf_{\theta\in (0,1/\Phi_\beta(\lambda))}\frac{\Psi_\beta(\lambda) \lambda c}{1-\Phi_\beta(\lambda) \theta} = \Psi_\beta(\lambda) \lambda c.
\]
Markov's inequality now gives for any $\lambda>0$,
\[
\P(F \geq \E F + r) = \P(e^{\lambda (F-\E F)} \geq e^{\lambda r}) \leq \frac{\E(e^{\lambda (F-\E F)})}{e^{\lambda r}} \leq \exp\left(\Psi_\beta(\lambda) \lambda c -\lambda r\right).
\]
Optimizing in $\lambda$ yields the desired deviation bounds.
\end{proof}
}

\begin{proof}[Proof of Corollary \ref{cor2}]
Here we adapt and combine the proofs of \cite[Theorem 8 and Theorem 9]{BLM_2003}. {For $\alpha=0$, the statement follows directly from Theorem \ref{EntIneqs}, so let $\alpha\in(0,2)$.} Let $\gamma = 1-\alpha/2$. Then, {on the event $\{F\neq 0\}$, we have}
\begin{align*}
&\int_\X (D^+_x F^\gamma(\eta-\delta_x))^2  d\eta(x)\\ &= \int_\X \ind\{F(\eta)^\gamma \geq F(\eta-\delta_x)^\gamma >0  \} (F(\eta)^ \gamma - F(\eta - \delta_x)^ \gamma)^2  d\eta(x) \\
& \quad\quad +  \int_\X \ind\{F(\eta)^\gamma \geq F(\eta-\delta_x)^\gamma { =0}  \} F(\eta)^ {2\gamma}  d\eta(x) 
\\ &= \int_\X \ind\{F(\eta) \geq F(\eta-\delta_x)  >0\}\left(\frac{F(\eta)}{F(\eta)^{1-\gamma}} - \frac{F(\eta - \delta_x)}{F(\eta - \delta_x)^ {1-\gamma}}\right)^2 d\eta(x)\\
&  \quad\quad+  \int_\X \ind\{F(\eta) \geq F(\eta-\delta_x) { =0}  \} F(\eta)^ {2\gamma}  d\eta(x).
\end{align*}
Since $1-\gamma >0$, we have that $F(\eta) \geq F(\eta - \delta_x)$ implies $F(\eta)^{1-\gamma} \geq F(\eta - \delta_x)^{1-\gamma}$. Hence, the above expression does not exceed
\begin{align*}
&\int_\X \ind\{F(\eta) \geq F(\eta-\delta_x)  >0\}\left(\frac{F(\eta)}{F(\eta)^{1-\gamma}} - \frac{F(\eta - \delta_x)}{F(\eta)^ {1-\gamma}}\right)^2  d\eta(x)\\
&  \quad\quad+  \int_\X \ind\{F(\eta) \geq F(\eta-\delta_x) { =0}  \} F(\eta)^ {2\gamma}  d\eta(x)\\
= \ &\frac{1}{F(\eta)^{\alpha}}\int_\X \left(D^+_x F(\eta - \delta_x)\right)^2 d\eta(x).
\end{align*}
Quite similarly one obtains {that on the event $\{F\neq 0\}$,}
\begin{align*}
\int_{\X} (D^-_x F^\gamma(\eta))^2 \ d\mu(x) \leq \frac{1}{F(\eta)^\alpha} \int_{\X} (D^-_x F(\eta))^2 \ d\mu(x).
\end{align*}
Hence, it follows that {on the event $\{F\neq 0, V^+ \leq G F^\alpha\}$,}
\begin{align*}
V^+(F^\gamma) = \int_{\X} (D^-_x F^\gamma(\eta))^2 \ d\mu(x) + \int_\X (D^+_x F^\gamma(\eta-\delta_x))^2  d\eta(x) \leq \frac{V^+}{F(\eta)^\alpha} \leq G.
\end{align*}
{Moreover, it is easy to check that on the event $\{F=0, V^+\leq GF^\alpha\}$, one has that $V^+(F^\gamma) = 0 = V^+$. Therefore, by virtue of the assumption that almost surely $V^+\leq GF^\alpha$, it follows that almost surely $V^+(F^\gamma)\leq G$.} Applying Theorem \ref{EntIneqs} to the random variable $F^\gamma$ yields the result.
\end{proof}

\begin{proof}[Proof of Corollary \ref{DevIneq2}]
{For $\alpha=0$, the statement follows directly from Corollary \ref{DevIneq1} (ii), so let $\alpha\in(0,2)$.} Let $\gamma = 1 - \alpha/2$. Continuing in the same way as in the proof of Corollary \ref{cor2} yields {that almost surely}
\begin{align*}
V^+(F^\gamma) = \int_{\X} (D^-_x F^\gamma(\eta))^2 \ d\mu(x) + \int_\X (D^+_x F^\gamma(\eta-\delta_x))^2 d\eta(x) \leq c.
\end{align*}
We conclude that Corollary \ref{DevIneq1} (ii) applies to $F^\gamma$. So $F^\gamma$ is non-negative and has an exponentially decaying upper tail. Thus, by virtue of \cite[Lemma 3.4]{K_2002}, all moments of $F^\gamma$ exist. In particular, $F$ is integrable. As it was pointed out in \cite[p. 1588]{BLM_2003}, we can now write
\begin{align*}
\P(F\geq\E F + r) &= \P(F^ \gamma \geq (r + \E F)^ \gamma) \leq \P(F^ \gamma - \E (F^ \gamma) \geq (r + \E F)^ \gamma - (\E F)^ \gamma)\\ 
&\leq \exp\left(-\frac{((r+\E F)^{\gamma} - (\E F)^{\gamma})^2}{2c}\right).
\end{align*}
\end{proof}

We continue with the proof of Theorem \ref{arbSignThm}. To get prepared for this, we first establish the following lemma.

\begin{lem}\label{truncLem}
Let $n\in\N$ and consider $F_n = \min(\max(F,-n),n)$ and $V^+_{(n)}=V^+(F_n)$. Then for any real number $b\geq 0$, almost surely
\begin{align*}
F (V^+_{(n)} - b) \leq F_n (V^+ - b) \ \ \text{if} \ \ F,F_n\geq 0,\\
F (V^+_{(n)} - b) \geq F_n (V^+ - b) \ \ \text{if} \ \ F,F_n\leq 0.\\
\end{align*}
\end{lem}
\begin{proof}
It is easy to see that $V^+_{(n)} \leq V^+$. Hence, the desired statement holds on the event $\{F=F_n\}$. If $F\neq F_n$, then either $F_n = n < F$ or $F_n = -n > F$. The latter case implies $V^+_{(n)} = 0$ and $F,F_n \leq 0$, hence the desired statement holds. So consider the case $F_n = n < F$ and let $A=F/n$. Then the desired inequality is equivalent to
\begin{align*}
A V^+_{(n)} \leq V^+ +(A-1) b.
\end{align*}
Since $b\geq 0$ and $A>1$, the above inequality is implied by $AV^+_{(n)}\leq V^+$, i.e. by
\begin{align*}
A\left(\int_\X (n-F_n(\eta - \delta_x))_+^2 d\eta(x) + \int_\X (F_n(\eta + \delta_x ) - n)_-^2 d\mu(x)\right)\\
 \leq \int_\X (An - F(\eta - \delta_x))_+^2 d\eta(x) + \int_\X (F(\eta + \delta_x) - An)_-^2 d\mu(x).
\end{align*}
To prove this, it suffices to conclude
\begin{align}
\label{ineq1}A(n-F_n(\eta-\delta_x))_+^2 \leq (An - F(\eta - \delta_x))_+^2,\\
\label{ineq2}A(F_n(\eta+\delta_x) - n)_-^2 \leq (F(\eta + \delta_x) - An)_-^2.
\end{align}
We prove (\ref{ineq1}). If $F(\eta - \delta_x)>n$, then 
\begin{align*}
A (n-F_n(\eta - \delta_x))_+^2 = 0 \leq (An-F(\eta - \delta_x))_+^2.
\end{align*}
If $F(\eta - \delta_x)\leq n$, then
\begin{align*}
F(\eta - \delta_x) \leq F_n(\eta - \delta_x) =: m.
\end{align*}
Now, since $|m| \leq n$, we have $A n^2 - m^2 \geq 0$. This gives $(A^2- A)n^2 + (1-A)m^2\geq 0$, thus
\begin{align*}
(An-m)^2 = A^2 n^2 - 2 Anm + m^2 \geq An^2 - 2Anm + Am^2 = A(n-m)^2.
\end{align*}
Hence,
\begin{align*}
A(n-F_n(\eta-\delta_x))_+^2 \leq
(An-F_n(\eta-\delta_x))_+^2 \leq
(An-F(\eta-\delta_x))_+^2,
\end{align*}
where the last inequality follows from $F(\eta - \delta_x) \leq F_n(\eta - \delta_x)<An$. This proves (\ref{ineq1}) and analogously one obtains (\ref{ineq2}). The result follows.
\end{proof}

\begin{proof}[Proof of Theorem \ref{arbSignThm}]
For the case when $F$ is bounded, we adapt the proof of \cite[Theorem 5]{BLM_2003}. Here we can argue in the same way as in the beginning of the proof of Theorem \ref{EntIneqs} to obtain for any $u\in(0,\lambda]$,
\begin{align*}
\Ent(e^{u F}) \leq \frac{1}{2} u^2 \E(V^+ e^{u F}).
\end{align*}
Invoking the assumption on $V^+$ yields
\begin{align*}
u \E(F e^{u F}) - \E(e^{u F})\log \E(e^{u F}) \leq \frac{1}{2} u^2 (a\E(F e^{u F}) + b\E(e^{u F})).
\end{align*}
With $h(u) = \tfrac{1}{u} \log \E(e^{u F})$ this can be rearranged as
\begin{align*}
h'(u) \leq \frac{1}{2}(a \log(\E (e^{u F}))' + b).
\end{align*}
Integrating this from $0$ to $\lambda$ gives
\begin{align*}
h(\lambda) - \E F \leq \frac{1}{2}\left( a \log(\E (e^{\lambda F}) +\lambda b\right).
\end{align*}
Noting that $a\lambda<2$ and rearranging the above inequality, we obtain the result for the bounded case.

For the unbounded case, consider for any $n\in\N$ the truncated random variables
\begin{align*}
F_n = \min(\max(F, -n), n).
\end{align*}
It follows from the assumptions and Lemma \ref{truncLem} that almost surely
\begin{align*}
F(V^+_{(n)} - b) \leq aF_nF \ \ \text{if} \ \ F\geq 0,\\
F(V^+_{(n)} - b) \geq aF_nF \ \ \text{if} \ \ F\leq 0.
\end{align*}
Note also that for $F = 0 = F_n$ we have $V^+_{(n)} \leq V^+ \leq b$. Therefore, almost surely
\begin{align*}
V^+_{(n)} \leq a F_n + b,
\end{align*}
so the result holds for all $F_n$. By dominated convergence, the sequence $\E F_n$ is convergent, hence bounded above by some constant $C$. Moreover, we can choose a $\nu> 1$ such that $\nu\lambda < 2/a$. Thus, since we already proved that the result applies to all the $F_n$, we conclude
\begin{align*}
\sup_{n\in\N}\E[\exp(\lambda(F_n-\E F_n))^\nu] \leq \exp \left(\frac{\nu^2\lambda^2}{2 - a\nu\lambda} (a C + b) \right) < \infty.
\end{align*}
By the Theorem of de la Vall\'ee-Poussin this implies that the family of random variables
\begin{align*}
\{\exp(\lambda(F_n-\E F_n))\}_{n\in\N}
\end{align*}
is uniformly integrable. Continuing as in the proof of Theorem \ref{EntIneqs} gives
\begin{align*}
\lim_{n\to\infty}\E \exp(\lambda(F_n-\E F_n)) = \E \exp(\lambda(F-\E F))<\infty.
\end{align*}
We note again that the result is already proved for the $F_n$ and that $\E F_n \to \E F$ as $n\to\infty$. This concludes the proof of the first inequality.

The deviation inequality now follows using the inequality we just proved together with Markov's inequality and \cite[Lemma 11]{BLM_2003}.
\end{proof}

For the proof of Theorem \ref{LowerDevIneq} we use the following FKG inequality for Poisson point processes, taken from \cite[Lemma 2.1]{J_1984}, see also \cite[Theorem 1.4]{LP_2011}.

\begin{lem} \label{FKG}
Let $F$ and $G$ be bounded Poisson functionals and assume that
\begin{align*}
D_xF(\xi), D_xG(\xi)\geq 0 \ \ \text{for all} \ \ (x,\xi)\in\X\times\OX.
\end{align*}
Then
\[
\E(FG)\geq (\E F)(\E G).
\]
\end{lem}

It was also remarked in \cite{J_1984} that under conditions like $F,G\geq 0$ or $\E F^2, \E G^2 < \infty$, the above result easily extends to unbounded functionals by monotone convergence. For our purpose we need the following extension.

\begin{cor} \label{FKGCor}
Let $F, G \geq 0$ be Poisson functionals. Assume that $F$ is bounded and $G$ is integrable. Moreover, assume that
\begin{align*}
D_xF(\xi)\leq 0 \ \ \text{and} \ \ D_xG(\xi)\geq 0 \ \ \text{for all} \ \ (x,\xi)\in\X\times\OX.
\end{align*}
Then
\[
\E(FG)\leq (\E F)(\E G).
\]
\end{cor}

\begin{proof}
Since $F$ is bounded, it follows from $\E G<\infty$ that also $\E(FG) < \infty$. Now consider for any $n\in\N$ the truncations $G_n = \min(G, n)$. Then we have almost surely $G_n\to G$ and $F G_n \to FG$ as $n\to \infty$. By monotone convergence, $\E G_n\to \E G$ and $\E(F G_n)\to\E(FG)$ as $n\to\infty$. It follows from Lemma \ref{FKG} that for any $n\in\N$,
\begin{align*}
\E(F G_n) \leq (\E F)(\E G_n).
\end{align*}
The result follows.
\end{proof}
\pagebreak

The following proof is inspired by ideas from the proof of \cite[Theorem 6]{BLM_2003}.

\begin{proof}[Proof of Theorem \ref{LowerDevIneq}]
For any $n\in\N$ consider the truncations
\begin{align*}
F_n = \min(\max(F, -n), n).
\end{align*}
Then the $F_n$ are again non-decreasing. Let $\lambda < 0$. It follows from Proposition \ref{EntIneq} with $I = \X\times\OX$ that for any $u\in[\lambda,0)$ we have
\begin{align*}
\Ent(e^{u F_n}) \leq \E \left[e^{u F_n} \int_{\X} \psi(u D_x F_n) \ d\mu(x)\right].
\end{align*}
Since $\psi(-z) \leq (1/2)z^2$ for $z\geq 0$, the right hand side of the above expression does not exceed
\begin{align*}
\frac{1}{2}\E \left[ e^{u F_n} \int_{\X} (u D_x F_n)^2 \ d\mu(x)\right] = \frac{1}{2} u^ 2 \ \E (e^{u F_n} V^-(F_n)).
\end{align*}
We have $V^-(F_n)\leq V^-$ almost surely, hence $\E(e^{uF_n}V^-(F_n))$ in the above display can be upper bounded by $\E(e^{uF_n}V^-)$. Now, since $F_n$ is non-decreasing and $u<0$, the functional $e^{u F_n}$ is non-increasing and bounded. Moreover, by assumption the functional $V^-$ is non-decreasing and $\E V^- <\infty$. Hence, by Corollary \ref{FKGCor} we have
\begin{align*}
\E (e^{u F_n} V^-) \leq \E (e^{u F_n})  \ \E V^-.
\end{align*}
It follows that
\[
h'(u)\leq \frac{1}{2}\E V^- \ \ \text{where} \ \ h(u) = \frac{1}{u} \log\E(e^{u F_n}).
\]
Integrating from $\lambda$ to $0$ yields
\begin{align*}
\log \E[\exp(\lambda(F_n-\E F_n))] \leq \frac{1}{2}\lambda^2 \E V^-.
\end{align*}
Since $\E V^- < \infty$, by Lemma \ref{ExpLem} we have $\E |F|<\infty$. Thus, applying the Theorem of de la Vall\'ee-Poussin similarly as in the proof of Theorem \ref{EntIneqs}, we conclude that the inequality in the last display also holds for the random variable $F$. Using Markov's inequality and optimizing in $\lambda$ yields the result.
\end{proof}

To prove the statement of Theorem \ref{vPlusLowerThm} for bounded $F$, we adapt the proof of the product space version \cite[Theorem 13]{M_2006}. To extend the result to unbounded $F$, Lemma \ref{truncLem} is used similarly as it was done in the proof of Theorem \ref{arbSignThm}.

\begin{proof}[Proof of Theorem \ref{vPlusLowerThm}]
First consider the case when $F$ is bounded. Let $\lambda < 0 $ and $u\in[\lambda,0)$. Then by Proposition \ref{EntIneq} {with $I=\emptyset$} we have
\begin{align*}
\Ent(e^{u F}) &\leq \E \left[ e^{u F}\int_\X \phi(-u D_x F(\eta-\delta_x)) d\eta(x)\right].
\end{align*}
Moreover, since $0\leq D F \leq 1$, we have $-u D_x F(\eta-\delta_x) \leq -u$ and since the map $z\mapsto \phi(z)/z^2$ is increasing, this implies
\begin{align*}
\int_\X \phi(-u D_x F(\eta-\delta_x)) d\eta(x) &=u^2 \int_\X \frac{\phi(-u D_x F(\eta-\delta_x))}{u^2 D_x F(\eta-\delta_x)^2} D_x F(\eta-\delta_x)^2 d\eta(x)\\
&\leq u^2 \int_\X \frac{\phi(-u) }{u^2} D_x F(\eta-\delta_x)^2 d\eta(x).
\end{align*}
Now, since $V^+ \leq aF$, we obtain
\begin{align*}
\Ent(e^{u F}) \leq \phi(-u) \E(e^{u F} V^+) \leq \phi(-u) a \E(Fe^{u F}).
\end{align*}
Dividing by $u^2\E(e^{u F})$ and integrating from $\lambda$ to $0$ yields
\begin{align*}
\E F - \frac{1}{\lambda} \log\E(e^{\lambda F}) \leq -\frac{\phi(-\lambda)}{\lambda^ 2} a \log\E(e^ {\lambda F}).
\end{align*}
{Since $1-a\phi(-\lambda)/\lambda>1$}, this can be rearranged as
\begin{align*}
\log\E[\exp(\lambda(F-\E F))] \leq \lambda^2 \frac{a\phi(-\lambda)/\lambda^2}{1- a\phi(-\lambda)/\lambda} \E F.
\end{align*}
Similarly as in the proof of Theorem \ref{arbSignThm}, the above inequality can be extended to the case when $F$ is unbounded. Here one should notice that according to Corollary \ref{DevIneq2}, the condition $V^+\leq aF$ guarantees $\E|F|<\infty$. It was pointed out in \cite{M_2006} that for any $\lambda<0$,
\begin{align*}
\frac{a\phi(-\lambda)/\lambda^2}{1- a\phi(-\lambda)/\lambda} \leq \frac{\max(a,1)}{2}.
\end{align*}
Markov's inequality now gives
\begin{align*}
\P(F\leq \E F - r) \leq \E(e^{\lambda(F-\E F)})e^{\lambda r} \leq \exp\left(\lambda^2 \frac{\max(a,1)}{2} \E F + \lambda r\right).
\end{align*}
Optimizing in $\lambda$ concludes the proof.
\end{proof}

\section{Applications to U-Statistics} \label{AppUstat}
{
\subsection{General remarks} The aim of the present section is to investigate the concentration properties of Poisson U-statistics. For this purpose, we need to specialize the very general framework that was in order so far. Throughout this section, the intensity measure $\mu$ on the space $\X$ is assumed to be non-atomic, that is, $\{x\}\in \mathcal{X}$ and $\mu(\{x\}) = 0$, for every $x\in \X$. This assumption is equivalent to the fact that the Poisson process $\eta$ on $\X$ is \emph{simple}, meaning that almost surely $\eta(\{x\})\leq 1$ for all $x\in\X$. It is common practice in this setting to identify the simple point process $\eta$ with its support, which now corresponds to a random set in $\X$. Plainly, the integral of a map $f:\X\to\R$ with respect to $\eta$ is now exactly given by the (possibly infinite) sum
\begin{align*}
\int_\X f \ d\eta = \sum_{x\in\eta} f(x).
\end{align*}

\begin{rem} \label{representativeRem}
Consider a Poisson functional $F$ together with some representative $\mathfrak{f}:\OX\to\R$. For any $\xi\in\OX$, we denote by $[\xi]$ the integer-valued measure uniquely determined by its value on singletons via the relation $[\xi](\{x\}) = \ind\{\xi(\{x\})>0\}$, for all $x\in\X$. Then, since $\eta$ is simple, we have that almost surely $F = \mathfrak{f}([\eta])$. It follows that another representative of $F$ is given by $\mathfrak{f}':\OX\to\R$, where $\mathfrak{f}'(\xi) = \mathfrak{f}([\xi])$. Therefore, without loss of generality, we can assume that $F(\xi)=F([\xi])$ for all $\xi\in\OX$, that is, given an arbitrary functional $\xi \mapsto F(\xi)$, in this section we will systematically select a representative of $F$ that only depends on $\xi$ via the mapping $\xi\mapsto [\xi]$. With this convention, one has that $D_xF(\xi) = D_x F([\xi])$, and also that $D_x F(\xi) = 0$ whenever $\xi(\{x\})>0$. Finally we observe that, again by virtue of the above convention and in accordance with the content of Remark \ref{r:talich}, the fact that the quantity $D_xF(\xi)$ verifies some property $\mathcal{P}$ for every $x\in \X$ and every $\xi \in {\bf N}$ is equivalent to the fact that $\mathcal{P}$ is verified for all $(x,\xi)\in \X\times\OX$ such that $\xi$ charges each singleton with a mass at most equal to 1.
\end{rem}

We now recall some relevant definitions. Let $f:\X^k \to \R_{\geq 0}$ be a symmetric measurable map and define the functional $S_f:\OX\to[0,\infty]$ by
\begin{align}\label{e:1}
S_f(\xi) = \sum_{\x\in\xi_{\neq}^k} f(\x).
\end{align}
A \emph{(Poisson) U-statistic} $F$ of order $k$ with kernel $f$ is a random variable such that almost surely $F = S_f(\eta)$. According to the Slyvniak-Mecke formula \eqref{e:smecke}, the expectation of a U-statistic $F$ is given by
$$
E[F] = \int_\X\cdots \int_\X f(x_1,...,x_m) d\mu(x_1) \cdots d\mu(x_m);
$$
see e.g. \cite[Section 3]{RS_2013} for more details as well as for an introduction to U-statistics with kernels that may have arbitrary sign.

\subsection{Choice of a representative} \label{s:representative}
In order to apply results from Section \ref{DIPF} to a Poisson U-statistic $F$ with kernel $f\geq 0$, we first need to choose a suitable representative of $F$ as defined in Section \ref{s:frame}. Whenever the considered U-statistic $F$ is almost surely finite, we can choose as a representative of $F$ the map $\mathfrak{f}:\OX\to\R$ defined by $\mathfrak{f}(\xi)=S_f(\xi)$ if $S_f(\xi)<\infty$ and $\mathfrak{f}(\xi)=0$ if $S_f(\xi)=\infty$.

In order to avoid technical problems arising from the choice of this representative, we will often assume that a given U-statistic with kernel $f$ is \emph{well-behaved}. By this we mean that there exists a measurable set $B\subseteq \OX$ with $\P(\eta\in B)=1$, such that
\begin{enumerate}
\item $S_f(\xi)<\infty$ for all $\xi\in B$,
\item $\xi+\delta_x\in B$ whenever $\xi\in B$ and $x\in\X$,
\item $\xi-\delta_x\in B$ whenever $\xi\in B$ and $x\in\xi$.
\end{enumerate}
If $F$ is well-behaved, then we will choose as a representative of $F$ the map $\mathfrak{f}:\OX\to \R$ defined by
$\mathfrak{f}(\xi) = S_f(\xi)$ if $\xi\in B$ and $\mathfrak{f}(\xi)=0$ if $\xi\in B^c$. Then, for any $(x,\xi)\in\X\times\OX$ one has $D_xF(\xi) = S_f(\xi+\delta_x) - S_f(\xi) < \infty$ if $\xi\in B$ and $D_xF(\xi) = 0$ if $\xi\in B^c$.

Note that by virtue of (\ref{xiNeq}) and (\ref{e:1}), the above choices of a representative imply $F(\xi) = F([\xi])$ for all $\xi\in\OX$ which is consistent with Remark \ref{representativeRem}. Finally, note that U-statistics that arise in typical applications (in particular, all $U$-statistics considered in this paper) are usually well-behaved in the sense described above.

\subsection{General results}
We will use an explicit expression for the difference operator of a U-statistic that was established in \cite{RS_2013}.
The following result gathers together several results from \cite[Lemma 3.3 and Theorem 3.6]{RS_2013}, in a form that is adapted to our setting.

\begin{prop} \label{DUstatProp} Let the above assumptions and notation prevail, let $F$ be a U-statistic with non-negative kernel $f$ and let $S_f$ be as in (\ref{e:1}). Then, for any $\xi\in\OX$ and $x\in\xi$, one has
\[
S_f(\xi)-S_f(\xi-\delta_x) = k F(x,\xi) \ \ \text{whenever} \ \ S_f(\xi)<\infty,
\]
where, for any $\xi \in {\bf N}$ and every $x\in\xi$ such that $\xi(\{x\})=1$, the \emph{local version} of $F$ is defined as
\begin{align}\label{e:genug}
F(x,\xi) := \sum_{\y\in(\xi\setminus x)^{k-1}_{\neq}} f(x,\y),
\end{align}
where $\xi\setminus x$ is shorthand for the set obtained by deleting $x$ from the support of $\xi$, and $F(x,\xi)=0$ whenever $\xi(\{x\})>1$. Moreover, if $\E F^2<\infty$, then $f\in  L^1(\mu^k)\cap L^2(\mu^k)$.
\end{prop}
{
As a direct consequence of the above result together with our canonical choices of a representative, described in Section \ref{s:representative}, we obtain:
\begin{cor}\label{DCor}
Let $F$ be a U-statistic with non-negative kernel $f$. Then the following statements hold:
\begin{enumerate}
\item If $F$ is almost surely finite, then there exists a measurable set $B\subseteq \OX$ that satisfies $\P(\eta\in B)=1$ such that for any $\xi\in B$ and $x\in\xi$, the local version $F(x,\xi)$ is finite and
\begin{align*}
D_xF(\xi - \delta_x) = k F(x,\xi).
\end{align*}
\item If $F$ is well-behaved, then there exists a measurable set $B\subseteq \OX$ that satisfies $\P(\eta\in B)=1$ such that the following holds:
\medskip

\begin{enumerate}
\item For any $\xi\in B$, $x\in\xi$ and $z\in\X$, the local versions $F(x,\xi)$ and $F(z,\xi+\delta_z)$ are finite, and moreover
\begin{align*}
D_xF(\xi-\delta_x) = k F(x,\xi) \ \ \text{and} \ \ D_zF(\xi) = kF(z,\xi+\delta_z).
\end{align*}

\item For any $\xi\in B^c$, $x\in\xi$ and $z\in\X$, one has
\begin{align*}
D_xF(\xi-\delta_x) = 0 = D_zF(\xi).
\end{align*}
\end{enumerate}
\end{enumerate}
\end{cor}
}

The previous Corollary \ref{DCor} implies that, if $F$ is an almost surely finite U-statistic with kernel $f\geq 0$, then almost surely
\begin{align*}
V^+ &= k^2 \sum_{x\in \eta} F(x,\eta)^2.
\end{align*}
If $F$ is in addition well-behaved, then almost surely
\begin{align*}
V^- &= k^2 \int_\X F(x,\eta+\delta_x)^2 d \mu(x).
\end{align*}
We have therefore the following consequences of Corollary \ref{DevIneq2} and Theorem \ref{LowerDevIneq}.
\begin{cor}
Consider an almost surely finite U-statistic $F$ of order $k$ with non-negative kernel $f$. Assume that for some $\alpha\in[0,2)$ and $c>0$ we have almost surely
\begin{align*}
\sum_{x\in \eta} F(x,\eta)^2 \leq c F^\alpha.
\end{align*}
Then $F$ is integrable and for all $r> 0$,
\begin{align*}
\P(F\geq\E F + r) \leq \exp\left(-\frac{((r+\E F)^{1-\alpha/2} - (\E F)^{1-\alpha/2})^2}{2ck^2}\right).
\end{align*}
\end{cor}

\begin{cor} \label{LowerDevIneqUStat}
Consider a well-behaved U-statistic $F$ of order $k$ with non-negative kernel $f$. Assume that
\begin{align}\label{e:2}
V := \E \int_\X F(x,\eta+\delta_x)^2 d \mu(x)<\infty.
\end{align}
Then, for all $r> 0$ we have
\begin{align*}
\P(F\leq\E F -r) \leq \exp\left(-\frac{r^2}{2k^2V}\right).
\end{align*}
\end{cor}

\begin{proof}
We have $\E V^-(F) = k^2 \E V < \infty$ and since $F$ is well-behaved, it follows from {Corollary \ref{DCor} (ii)} that $D_xF(\xi)\geq 0$ for any $(x,\xi)\in\X\times\OX$. So the result follows from Theorem \ref{LowerDevIneq} together with Remark $\ref{incVRem}$ once we proved that $DD F \geq 0$. According to \cite{RS_2013}, and since $F$ is well-behaved, for any $(z,x,\xi)\in\X\times\X\times\OX$, the second iteration of the difference operator either satisfies $D_zD_xF(\xi) = 0$ or it can be written as
\begin{align*}
D_zD_x F(\xi) = k(k-1) \sum_{\y\in\xi^{k-2}_{\neq}} f(z,x,\y).
\end{align*}
The right-hand side of the above display is non-negative since $f\geq 0$.
\end{proof}

}

\subsection{Computing $V$ in formula \eqref{e:2}}\label{ss:zut}

We will now provide a direct proof that condition \eqref{e:2} is equivalent to the fact that $F$ is a square-integrable U-statistic and that one can obtain a rather explicit expression for $V$ in terms of some set of auxiliary kernels built from $f$.

\begin{defi}\label{d:fi} Let $f$ be a symmetric element of $L^1(\mu^k)$, for some $k\geq 1$. For $i=1,...,k$, we define the kernels $f_i$ as follows: 
\begin{equation}\label{e:kw}
f_i(y_1,...,y_i) := \binom{k}{i} \int_{\X^{k-i}} f(y_1,...,y_i, z_1,...,z_{k-i}) d\mu^{k-i}(z_1,...,z_{k-i}),
\end{equation}
if the integral on the right-hand side is well defined, and $f_i(y_1,...,y_i)= 0$ otherwise. Observe that, since $f$ is in $L^1(\mu^k)$, then the class of those $(y_1,...,y_i)$ such that the integral on the right-hand side of \eqref{e:kw} is not defined has measure $\mu^i$ equal to zero, for every $i=1,...,k$. Plainly, each $f_i$ is a symmetric mapping from $\X^i$ into $\R$ and $f_i \in L^1(\mu^i)$, for every $i=1,...,k$,  and $f_k = f$ by definition. \end{defi}

The upcoming result provides new necessary and sufficient conditions for the square-integrability of U-statistics. {Although the investigations in the present paper (and hence also in the result below) are restricted to U-statistics with non-negative kernels, we stress that this assumption is not needed in the forthcoming proof, and thus, after appropriately adapting the notion of a well-behaved U-statistic for kernels with arbitrary sign, the presented characterization for square-integrable U-statistics also applies when the kernels are not necessarily assumed to be non-negative.}

\begin{prop}[(Characterization of square-integrable $U$-statistics)]\label{p:p} Consider a {well-behaved} U-statistic $F$ of order $k\geq 1$, with {non-negative} kernel $f\in L^1(\mu^k)$. Then, the following assertions are equivalent:
\begin{enumerate}

\item $F$ is square-integrable;

\item for every $i=1,...,k$, $f_i\in L^2(\mu^i)\cap L^1(\mu^i) $, where the kernels $f_i$ have been introduced in Definition \ref{d:fi};

\item $V<\infty$, where $V$ is defined in \eqref{e:2}.

\end{enumerate}
If either one of conditions {\rm (i)}, {\rm (ii)} or {\rm (iii)} is verified, then 
\begin{equation}\label{e:vv}
k^2\,  V = \sum_{i=1}^k ii! \| f_i\|^2_{L^2(\mu^i)} \quad\mbox{and}\quad \V F =  \sum_{i=1}^k i! \| f_i\|^2_{L^2(\mu^i)},
\end{equation}
so that, in particular, $V\leq k^{-1} \times \V F$.
\end{prop}

\begin{proof}

{}[Step 1: (i) $\to$ (ii), (iii) ] According to \cite[Theorem 3.6]{RS_2013}, if $F$ is a U-statistic as in the statement and $F$ is square-integrable, then necessarily $f_i\in  L^1(\mu^{i})\cap L^2(\mu^{i})$ for every $i=1,...,k$, and moreover $F$ admits the following representation:
$$
F=\E F + \sum_{i=1}^k I_i(f_i),
$$
where $I_i$ denotes a multiple Wiener-It\^o integral of order $i$, with respect to the compensated Poisson measure $\hat{\eta} = \eta-\mu$ (see e.g. \cite[Chapter 5]{PT_2010} for definitions). Note that, exploiting the standard orthonormality properties of multiple integrals, one has also that
$$
\V F =  \sum_{i=1}^k i!  \| f_i\| ^2_{L^2(\mu^i)},
$$
which corresponds to the first relation in \eqref{e:vv}. Combining \cite[Theorem 3.3]{LP_2011} with the previous discussion, one also infers that, if $F$ is square-integrable, then a version of the add-one cost operator $DF$ is given by
\begin{equation}\label{e:gum}
D_xF = \sum_{i=1}^k i I_{i-1} (f_i(x, \cdot) ),
\end{equation}
where $I_{i-1}(f_i(x, \cdot))$ indicates a multiple Wiener-It\^o integral of order $i-1$, with respect to $\hat\eta = \eta-\mu$, of the kernel $f_i(x,\cdot) : \X^{i-1} \to \R$, obtained from $f_i$ (see Definition \ref{d:fi}) by setting one of the variables in its argument equal to $x$; observe that, as usual, the right-hand side of \eqref{e:gum} is implicitly set equal to zero on the exceptional set of those $x\in \X$ such that $f_i(x, \cdot)\notin L^2(\mu^{i-1})$ for at least one $i\in \{1,...,k\}$.  Exploiting the standard orthonormality properties of multiple integrals, one has therefore that
\begin{equation}\label{e:zo}
\E[(D_xF)^2] = \sum_{i=1}^k i^2 (i-1)! \| f_i(x,\cdot)\|^2_{L^2(\mu^{i-1})}, 
\end{equation}
so that the conclusion (as well as the explicit expression of $k^2V =\E \int_\X (D_xF)^2 d\mu(x)$ appearing in \eqref{e:vv}) follows from an application of the Fubini Theorem.
\smallskip

[Step 2: (ii) $\to$ (i)] Assume that, for every $i=1,...,k$, $f_i\in L^2(\mu^i)\cap L^1(\mu^i)$. Then, according to \cite[Theorem 4.1]{surg} the multiple integral $I_i(f_i)$ is a well-defined square-integrable random variable, and moreover 
$$
I_i(f_i) = \sum_{j=0}^i(-1)^{i-j} \binom{i}{j} \sum_{(x_1,...,x_j)\in \eta^j_{\neq}} f_i^{(j)}(x_1,...,x_j),
$$ 
where $f_i^{(j)} :=\binom{k}{i} \times \binom{k}{j} ^{-1}\times  f_j$, and each (possibly infinite) sum in the previous expression converges in $L^1(\P)$. Now write 
$$
u(j) := \sum_{(x_1,...,x_j)\in \eta^j_{\neq}} f_{j}(x_1,...,x_j), \quad j=0,...,k.
$$
The previous discussion yields that
\begin{eqnarray*}
 \sum_{i=0}^k I_i(f_i) &=&  \sum_{i=0}^k \binom{k}{i} \sum_{j=0}^i(-1)^{i-j} \binom{i}{j} u(j)\\
& =& \sum_{j=0}^k u(j)\left\{ \sum_{i=j}^k \binom{k}{i}\binom{i}{j} (-1)^{i-j}\right\} = u(k)=F,
\end{eqnarray*}
where we have used the fact that the sum $\sum_{i=j}^k \binom{k}{i}\binom{i}{j} (-1)^{i-j}$ equals one if $j=k$, and vanishes otherwise. It follows that $F$ is square-integrable, since it is equal to a finite sum of square-integrable random variables. 

\smallskip

[Step 3: (iii) $\to$ (ii)] If $V<\infty$, then there exists a measurable set $B\subset \X$ such that $\mu(B^c)=0$, and $\E(D_xF)^2<\infty$, for every $x\in B$. {Using \cite[Lemma 3.5, Theorem 3.6]{RS_2013} together with the {fact that, since $F$ is well-behaved,} $D_xF$ is the {(well-behaved)} U-statistic of order $k-1$ defined in Proposition \ref{DUstatProp}}, we immediately deduce that, for $x\in B$, one has that (adopting the same notation as in Step 1) $f_i(x, \cdot) \in L^2(\mu^{i-1})$, and also
$$
D_xF = \sum_{i=1}^k i I_{i-1} (f_i(x, \cdot) ).
$$
The conclusion follows by using once again \eqref{e:zo} and the Fubini Theorem.
\end{proof}

\begin{rem} According e.g. to \cite[Lemma 3.1]{PT_alea}, the condition $\E \int_\X (D_xF)^2 d\mu(x)<\infty$ is equivalent to the fact that $F$ belongs to the domain of the {\it Malliavin derivative} associated with $\eta$. This fact is consistent with the fact that {square-integrable} $U$-statistics have a finite Wiener-It\^o chaotic expansion, and therefore belong automatically to the domain of the Malliavin derivative.
\end{rem}

An application of Proposition \ref{p:p} to the estimation of lower tails for edge-counting in random geometric graphs (involving in particular U-statistics of order $k=2$) is presented in Section \ref{ss:eclower}.


\section{Applications to edge counting}\label{s:edge}

In this section, we let $\eta$ denote a Poisson point process on $(\R^d, \mathcal{B}(\R^d))$, with intensity given by a Borel measure $\mu$ (in particular, $\mu(K)<\infty$ for every compact set $K$). {We also assume again} that $\mu$ has no atoms, that is, $\mu(\{x\}) = 0$ {for every $x\in \R^d$}. For a fixed $\rho>0$, we shall consider the graph $\mathfrak{G}$ (often called the {\it Gilbert graph}, or the {\it disk Graph} with radius $\rho$ associated with $\eta$) obtained as follows: the vertex set of $\mathfrak{G}$ is given by the points in the support of $\eta$, and two vertices $x, y$ are linked by an edge {(in symbols, $x\leftrightarrow y$)} whenever $0< \lVert x-y \rVert \leq \rho$ (in particular, $\mathfrak{G}$ has no loops). For technical reasons clarified below, we will assume for the rest of the section that the following condition on $\mu$ is verified: denoting $B(x,\rho)$ the {closed} ball of radius $\rho$ centered in $x$,
\begin{equation}\label{e:r}
\int_{\R^d} \, \mu(B(x,\rho)) \, d\mu(x)<\infty.
\end{equation}
Relation \eqref{e:r} is verified whenever $\mu(\R^d)<\infty$, but such a finiteness condition is not necessary for \eqref{e:r} to hold \footnote{ Consider for instance the measure $\mu$ on $\R^2$ having density $p(x)=(\lVert x\rVert + 1)^{-2}$ together with an arbitrary radius $\rho>0$}. Note that, if $\mu$ is Borel and \eqref{e:r} is in order, then the mapping $x\mapsto \mu(B(x,\rho))$ is necessarily bounded. To see this, choose $\gamma>0$ such that the ball $B(0,\rho)$ can be written as a union of $\lfloor 1/\gamma\rfloor$ many sets with diameter less than $\rho$. Then the pigeonhole principle yields that for any $y\in\R^d$ we can choose a set $C_y\subset B(y,\rho)$ satisfying: (i) $C_y\subseteq B(x,\rho)$ for all $x\in C_y$, and (ii) $\mu(C_y) \geq \gamma \mu(B(y,\rho))$. Now,
\begin{align*}
\int_{\R^d} \mu(B(x,\rho)) d\mu(x) &\geq \sup_{y\in\R^d} \int_{C_y} \mu(B(x,\rho))d\mu(x) \geq \gamma \sup_{y\in\R^d} \int_{C_y} \mu(B(y,\rho)) d\mu(x)\\
& = \gamma \sup_{y\in\R^d} \mu(C_y) \mu(B(y,\rho)) \geq \gamma^2 \sup_{y\in\R^d} \mu(B(y,\rho))^2.
\end{align*}

\medskip

Originally introduced in 1959 by Gilbert in the seminal work \cite{G_1959}, the disk graph $\mathfrak{G}$ is the archetypical example of a {\it random geometric graph}. Since then, the study of such an object has been at the center of a formidable collective effort, both at a theoretical and applied level. We refer the reader to the fundamental monograph \cite{Penrose_book} for a detailed overview of the literature on Gilbert graphs up to the year 2003. Recent developments that are relevant for our work are discussed e.g. in \cite{BP_2012, DFRV, LRP1, LRP2, RS_2013, RST_2013, ST_spa2012}.

\medskip

In this section, we will provide new concentration estimates for the random variable $$N =N(\eta):= \# \big\{\{x,y\}\subseteq \eta : x\leftrightarrow y\big\}, $$ corresponding to the number of edges of $\mathfrak{G}$. It is immediately seen that $N$ is a Poisson U-statistic of order $2$ with positive kernel $f(x,y) = \tfrac{1}{2} \ind\{\lVert x-y \rVert \leq \rho\}$. In particular, the Slivniak-Mecke formula \eqref{e:smecke} together with a standard use of the Fubini Theorem yields that the assumption \eqref{e:r} is actually equivalent to integrability of $N$ and that
$$
 \E N =\frac12 \int_{\R^d} \mu(B(x,\rho)) \, d\mu(x).
$$
We also see that assumption \eqref{e:r} implies that $N<\infty$ almost surely, yielding in turn that $N$ is well-behaved.

\subsection{Preparation: optimal rates} \label{optimalExp}

Let the above notation and assumptions prevail. In the forthcoming Section \ref{ss:rggut}, we will provide estimates for the upper tail of $N$ having the form
\begin{equation}\label{e:ir}
\P(N\geq \E N + r)\leq \exp(-I(r)), \quad r>0,
\end{equation}
where $r\mapsto I(r)$ is a positive mapping verifying 
\begin{equation}\label{e:rc}
\lim_{r\to\infty}I(r) = \infty.
\end{equation}
The next statement contains a universal necessary condition on the asymptotic behaviour of $I(r)$.

\begin{prop}\label{p:ny}
Let $I(r)$ verify \eqref{e:ir} and \eqref{e:rc}. Then,
$$
\limsup_{r\to \infty} \frac{I(r)}{r^{1/2} \log r}\leq \frac{1}{\sqrt{2}}.
$$
\end{prop}
\begin{proof} Let $x\in\R^d$ such that $m = \mu(B(x,\rho/2))>0$. Then $\hat{N} = \eta(B(x,\rho/2))$ is Poisson distributed with expectation $m$. Moreover, the distance between any $y,z\in B(x,\rho / 2)$ is at most $\rho$, thus any two vertices in $B(x,\rho / 2)$ are connected by an edge. This implies that almost surely $\hat{N}(\hat{N}-1)/2 \leq N$. Hence, for any $r\geq 0$ we have
\begin{align*}
\P(N\geq \E N + r) \geq \P(\hat{N}^2 - \hat{N} \geq 2\E N + 2r) \geq \P(\hat{N}\geq \gamma(r)),
\end{align*}
where $\gamma(r) = \sqrt{2\E N + \tfrac{1}{4} + 2 r} + \tfrac{1}{2}$. It is well known that for a Poisson random variable $X$ the upper tail satisfies $\P(X\geq r) \sim \exp(-r\log(r/\E X) - \E X)$ as $r\to \infty$, see for example \cite{G_1987}. Hence,
\begin{align*}
\liminf_{r\to\infty} \frac{\P(N\geq \E N + r)}{\exp(-\gamma(r)\log(\gamma(r)/m) - m)} \geq 1.
\end{align*}
The above considerations yield that {there exists a constant $C\geq 0$ such that}, for $r$ large enough along any subsequence diverging to infinity,
\begin{align*}
\gamma(r)\log(\gamma(r)/m)+m \geq I(r) - C.
\end{align*}
Dividing this inequality by $I(r)$ and letting $r$ diverge to infinity gives
\begin{align}\label{e:g}
\liminf_{r\to\infty} \frac{\gamma(r)\log(\gamma(r)/m)}{I(r)} \geq 1.
\end{align}
The conclusion is obtained by observing that, as $r\to\infty$,
\begin{align*}
\gamma(r)\log(\gamma(r)/m) \sim (r/2 )^{1/2}  \log r.
\end{align*}
\end{proof}

The following statement is an elementary consequence of Proposition \ref{p:ny}.

\begin{cor}\label{c:o} Let $r\mapsto I(r)$ be a positive mapping verifying \eqref{e:ir}, and assume that there exist constants $a, b>0$ such that, as $r\to\infty$, $I(r)\sim b\, r^a $. Then, necessarily, $a \leq \frac12$.

\end{cor}

\subsection{Deviation inequalities for the upper tail}\label{ss:rggut}
We will now deal with bounds on the upper tail of $N$. We start by observing that, for every $x\in \eta$, the local version $N(x, \eta)$, as defined in \eqref{e:genug}, is exactly given by {the quantity $\deg(x)/2$, where $\deg(x)=\#\{y\in\eta:x\leftrightarrow y\}$ is the \emph{degree} of the vertex $x$}. Our aim in what follows is to show that, for some constant $c>0$, one has that almost surely
\begin{align}\label{edgeIneq}
\sum_{x\in\eta} \deg(x)^ 2 \leq c N^ {3/2}.
\end{align}
Hence, Theorem \ref{DevIneq2} yields the following deviation inequality for the upper tail:
\begin{equation}\label{e:copa}
\P(N\geq\E N + r) \leq \exp\left(-\frac{((r+\E N)^{1/4} - (\E N)^{1/4})^2}{ 2 c}\right).
\end{equation}

Observe that the right-hand side of \eqref{e:copa} has the form $\exp(-I(r))$, where $I(r)\sim r^{1/2}/2c$, as $r\to\infty$. According to Corollary \ref{c:o}, the power $1/2$ for $r$ is optimal in this situation. We will see in Section \ref{ss:lpf} that, by adopting an alternative approach, the rate of decay of $I(r)$ can indeed be improved by the square root of a logarithmic factor.

Also notice that by virtue of relation \eqref{edgeIneq} together with Theorem \ref{DevIneq2}, almost sure finiteness of $N$ is equivalent to integrability of $N$. Hence, relation \eqref{e:r} actually holds if and only if $N$ is almost surely finite.
\medskip

We start by proving a geometric lemma, focussing on deterministic point configurations. In what follows, we shall write $\mathfrak{p} = \mathfrak{p}(d)$ to indicate the smallest integer $\mathfrak{p}$ such that the half ball $B=\{x\in\R^d: \lVert x\rVert \leq \rho, x_1>0\}$ can be written as a union $B = B_1\cup\ldots \cup B_{\mathfrak{p}}$ of disjoint sets such that $\diam(B_i)\leq \rho$ for all $i=1,...,\mathfrak{p}$. Note that the value of $\mathfrak{p}$ depends on the dimension $d$ of the surrounding Euclidean space. In the plane $\R^2$ one has for example that $\mathfrak{p} = 3$. The picture below illustrates the situation described in the proof of the upcoming lemma.
\medskip

\begin{figure}[ht]
\includegraphics[scale=0.3]{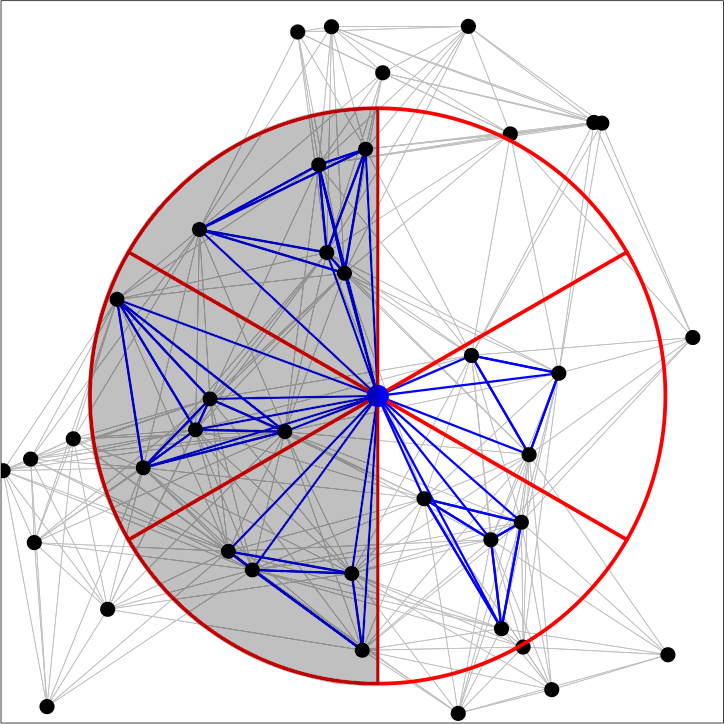}
\caption{Partitioning of the half ball}
\end{figure}

\begin{lem} \label{edgeLem}
Let $\xi\subset\R^d$ be a countable set. Denote the disk graph with radius $\rho>0$ associated with $\xi$ by $\mathfrak{G}_\xi$. For all $x\in\xi$, define the \emph{right-degree} and the \emph{left-degree}  of $x$ as
\[
\deg_r(x) = \#((x+B)\cap \xi) \ \ \text{ and } \ \ \deg_l(x) = \#((x-B)\cap \xi).
\]
Let $T_\xi$ and $N_\xi$ denote the number of triangles and edges in $\mathfrak{G}_\xi$, respectively. Then
\[
2\mathfrak{p} T_\xi + \mathfrak{p} N_\xi \geq \sum_{x\in\xi}  \deg_r(x)^2.
\]
This inequality also holds for the left-degree instead of the right-degree.
\end{lem}

\begin{proof}
For the rest of the proof, we write $\mathfrak{G}, \, T$ and $N$, without the subscript $\xi$, to simplify the notation. Without loss of generality, we can assume that $\xi$ only contains a finite number of non-isolated points in the topology of the graph; otherwise $N=\infty$ and the estimate in the statement is trivially satisfied. Let $x\in\xi$ and denote by $T(x)$ the number of those triangles $\{x,y,z\}$ in $\mathfrak{G}$ incident to $x$, and such that $x_1<\min(y_1,z_1)$, where $a_1$ indicates the first coordinate of a given vector $a\in\R^ d$. Moreover, for $i=1,\ldots,\mathfrak{p}$ let $n_i$ denote the number of elements of $\xi$ contained in $B_i+x$. Then, since any two vertices contained in the same $B_i+x$ yield an edge and thus a triangle incident to $x$ in the sense described above, we have
\[
T(x)\geq \sum_{i=1}^{\mathfrak{p}} \binom{n_i}{2}.
\]
In view of the relation $\sum_i n_i = \deg_r(x)$, the right-hand side of the previous expression can be further bounded from below, thus yielding the relation:
\[
T(x) \geq \frac{1}{2}\mathfrak{p} \cdot \frac{\deg_r(x)}{\mathfrak{p}}\left(\frac{\deg_r(x)}{\mathfrak{p}}-1\right).
\]
Hence,
\[
2\mathfrak{p} T(x) \geq \deg_r(x)^2-\mathfrak{p}\deg_r(x).
\]
Summing up over all $x\in\xi$ yields
\[
2\mathfrak{p} \sum_{x\in\xi} T(x) + \mathfrak{p}\sum_{x\in\xi}\deg_r(x) \geq \sum_{x\in\xi}\deg_r(x)^2.
\]
Observe that in the sum $\sum_{x\in\xi} T(x)$ each triangle is counted at most once, thus $\sum_{x\in\xi} T(x)\leq T$. Also, in the sum $\sum_{x\in\xi}\deg_r(x)$ each edge is counted at most once, so this sum is less than $N$.
\end{proof}
We now deduce a bound on $\sum_{x\in\xi}\deg(x)^2$ for any countable point configuration $\xi$.

\begin{cor} \label{edgeCor}
Let $\xi\subset\R^d$ be countable. Let $\mathfrak{G}_\xi$ be the disk graph (with some arbitrary radius $\rho>0$) associated with $\xi$ and denote the number of edges of $\mathfrak{G}_\xi$ by $N_\xi$. Then,
\[
\sum_{x\in\xi}\deg(x)^2 \leq \frac{8 \sqrt{2}}{3} \mathfrak{p} N_\xi^{3/2} +4\mathfrak{p}N_\xi
\leq \left(\frac{8 \sqrt{2}}{3} + 4\right) \mathfrak{p} N_\xi^{3/2}.
\]
\end{cor}

\begin{proof}
First we observe that without loss of generality it can be assumed that the first coordinates of all elements in $\xi$ are distinct. Obviously, the combinatorial structure of $\mathfrak{G}_\xi$ is invariant under rotation of the set $\xi$. Also, the assumption in question can be achieved to hold by rotating the set $\xi$ with respect to a direction $a\in\R^d, \lVert a \rVert = 1$ that satisfies $(x-y) / \lVert x-y \rVert \neq \pm a$ for all distinct $x,y\in \xi$. Such a direction exists since the set of directions
\begin{align*}
\{(x-y)/ \lVert x-y \rVert : (x,y)\in\xi_{\neq}^2\}
\end{align*}
is countable and hence a strict subset of all directions $\{a\in\R^d:\lVert a\rVert = 1\}$. So, it can be assumed that the elements of $\xi$ have distinct first coordinates. In particular, for any $x\in\xi$ we have $\deg(x) = \deg_r(x) + \deg_l(x)$. Thus
\[
\sum_{x\in\xi}\deg(x)^2 = \sum_{x\in\xi}(\deg_r(x)+\deg_l(x))^2 \leq
2 \sum_{x\in\xi} (\deg_r(x)^2 + \deg_l(x)^2).
\]
By Lemma \ref{edgeLem}, the latter expression does not exceed
\[
8 \mathfrak{p} T_\xi + 4\mathfrak{p} N_\xi,
\]
where $T_\xi$ stands for the number of triangles in $\mathfrak{G}_\xi$. The result follows by using an estimate taken from \cite{R_2002}, where it was proven that
\[
3T_\xi \leq \sqrt{2}N_\xi^{3/2}.
\]
\end{proof}

The next statement is one of the main achievements in the present section.

\begin{thm}
\begin{enumerate}

\item Let $\xi\subset\R^d$ be countable. Let $\mathfrak{G}_\xi$ be a disk graph with arbitrary radius $\rho$ associated with $\xi$, and denote the number of edges of $\mathfrak{G}_\xi$ by $N_\xi$. Then,
\[
\sum_{x\in\xi}\deg(x)^2 \leq \left(\frac{8 \sqrt{2}}{3} + \frac{4}{D}\right) \mathfrak{p} N_\xi^{3/2},
\]
where
\[
D=\frac{4\sqrt{2}}{3} \mathfrak{p} + \sqrt{\frac{32}{9}\mathfrak{p}^2 + 4\mathfrak{p} - 1}.
\]

\item Now let $\eta$ be the Poisson measure on $\R^d$ with non-atomic intensity $\mu$ considered in this section, and denote by $\mathfrak{G}_\eta = \mathfrak{G}$ the random disk graph (with arbitrary radius $\rho$) associated with $\eta$. Let $N_\eta = N$ be the number of edges of {the random disk graph $\mathfrak{G}$. Then relation \eqref{edgeIneq} holds almost surely, for $c=\left(\frac{8 \sqrt{2}}{3} + \frac{4}{D}\right) \mathfrak{p}(d)$.} In particular, the tail estimate \eqref{e:copa} is verified.

\end{enumerate}
\end{thm}

\begin{proof}
\noindent{[Proof of (i)}] By Corollary \ref{edgeCor} we have
\[
\sum_{x\in\xi}\deg(x)^2 \leq \left(\frac{8 \sqrt{2}}{3} +\frac{4}{\sqrt{N_\xi}}\right) \mathfrak{p}N_\xi^{3/2}.
\]
Observe that among all graphs with $N_\xi$ edges, the star, i.e. the graph with $\deg(x) = N$ for one vertex $x$ and $\deg(y)=1$ for all other vertices $y$, maximises the sum of the squared degrees. Thus, we also have
\[
\sum_{x\in\xi} \deg(x)^2 \leq N_\xi^2 + N_\xi = \left(\sqrt{N_\xi} + \frac{1}{\sqrt{N_\xi}}\right) N_\xi^{3/2}.
\]
Now, since 
\[
\left(\frac{8 \sqrt{2}}{3} +\frac{4}{\sqrt{N_\xi}}\right) \mathfrak{p}
\]
is monotonically decreasing in $N_\xi$ and
\[
\sqrt{N_\xi} + \frac{1}{\sqrt{N_\xi}}
\]
is monotonically increasing in $N_\xi$,
the minimum of both functions is always less or equal to the value at the intersection of them. Computing this value yields the result.
\smallskip

\noindent{[Proof of (ii)}] This follows directly from part (i) of the statement.
\end{proof}

\subsection{Deviation inequalities for the lower tail}\label{ss:eclower} We now focus on lower tails. In order to do that, we introduce the notation
$$
K := \sup_{x\in \R^d} \mu(B(x, \rho))<\infty,
$$ 
and we define the parameter ${\mathfrak v}$ as
$$
{\mathfrak v}:= 2(K+1)\, \E N.
$$
Our main estimate is the following:

\begin{thm}
For every $r>0$ one has the estimate
\begin{align}\label{e:gglt}
\P(N\leq\E N -r) \leq \exp\left(-\frac{r^2}{2{\mathfrak v}}\right).
\end{align}
\end{thm}
\begin{proof} We will freely use the notation and definitions introduced in Section \ref{ss:zut}. Since $N$ is a U-statistic of order $2$ with kernel $f(x,y) = \tfrac{1}{2} \ind\{\lVert x-y \rVert \leq \rho\}$, by virtue of Proposition \ref{p:p} it is sufficient to show that, in this case, $4 V\leq \mathfrak{v}$. In order to do that, observe first that $f_1(x) = \mu(B(x,\rho))$ and $f_2=f$. As a consequence, 
\begin{align*}
&\|f_1\|^2_{L^2(\mu^1)} = \int_{\R^d} \mu(B(x,\rho))^2 \, d\mu(x) \leq 2 K \E N;\\
&\|f_2 \|^2_{L^2(\mu^2)} = \tfrac 12\E N.
\end{align*}
Plugging the above upper bounds into the definition of $V$ yields the inequality $4V \leq \mathfrak{v}$, and therefore the desired conclusion.
\end{proof}

\subsection{Comparison with the literature}
We shall now briefly compare the concentration inequalities for edge counting presented above with the results already existing in the literature. To the best of our knowledge, the only concentration inequalities known so far that apply to the setting of edge counting in random disk graphs over Poisson point configurations, are established in \cite{RST_2013} and \cite{ERS_2015}. Both papers deal with the case where the {intensity measure $\mu$ is finite.}

\smallskip

\noindent{\it Comparison with {\rm \cite{ERS_2015}}}. We shall use some computations from \cite{RST_2013}, where it is explained how the general results about stabilizing functionals from \cite{ERS_2015} apply to random graph statistics. In the special case of edge counting, using the notation and assumptions of the present section, one deduces indeed from \cite[Proposition 5.1]{RST_2013} an estimate of the type
$$
\P( |N - \E N| \geq r) \leq \exp(-I_0(r)), \quad \mbox{where } I_0(r) \sim a r^{1/3},
$$
for some positive constant $a$, as $r\to\infty$. As far as the asymptotic behaviour of $r\mapsto I_0(r)$ is concerned, this result is worse than the estimates that one can obtain from \eqref{e:copa} (where the argument of the exponential bound on the upper tail is asymptotic to  $-r^{1/2}/2c$, and therefore optimal in the sense of Corollary \ref{c:o}), and than those given by \eqref{e:gglt} (where we have proved a Gaussian upper bound on the lower tail).

\smallskip
 
\noindent{\it Comparison with {\rm \cite{RST_2013}}.} Writing $m$ for a median of the law of $N$, in \cite[Theorem 5.2]{RST_2013} one can find an estimate of the type
$$
\P( |N - m | \geq r) \leq \exp(-I_1(r)), \quad\mbox{where }\, I_1(r)\sim b r^{1/2},
$$
as $r\to \infty$, for some positive constant $b$. Note that, as $r\to\infty$, the upper tail bound determined by $I_1$ has the same order as our estimate \eqref{e:copa}, whereas the lower tail estimate is worse than our Gaussian upper bound \eqref{e:gglt}. We also stress that \cite[Theorem 5.2]{RST_2013} has a different nature than our results, since it gives a concentration inequality around the median, and not the expectation. This might be a drawback for applications since the median of the edge count is harder to deal with than the expectation -- which can be easily expressed using the Slivniak-Mecke formula. Finally, a major advantage of our results over those presented in \cite[Theorem 5.2]{RST_2013} is that the latter only applies to disk graphs built on finite intensity measure Poisson processes, whereas our tail estimates merely require that {the number of edges is almost surely finite.}

\subsection{Consistency with CLT}
In the following, we compare the deviation inequality for the upper tail with a CLT that was proven in \cite{RS_2013}. Let $\eta_1$ be a Poisson point process in $\R^d$ with intensity measure $\mu_1$ and fix some radius $\rho > 0$. For any $n\in\N$, let $\eta_n$ be a Poisson point process with intensity measure $\mu_n = n\mu_1$ and denote the number of edges in the corresponding random geometric graph by $N_n$. Assume that $\E N_1 < \infty$ (and hence $\E N_n < \infty$ for all $n$). Then by \cite[Theorem 5.2]{RS_2013}, the sequence of random variables $N_n$ satisfies a central limit theorem, i.e. $(N_n - \E N_n) / \sqrt{\V N_n}$ converges to a standard Gaussian distribution, where $\V N_n$ stands for the variance of $N_n$. Therefore, as $n\to\infty$, the sequence of probabilities $\P(N_n\geq \E N_n + \sqrt{\V N_n} r)$, $n\geq 1$, converges to the quantity 
$$
\frac{1}{\sqrt{2\pi}}\int_r^\infty e^{-y^2/2} dy\leq \frac{e^{-r^2/2}}{r\sqrt{2\pi}}.
$$ 
According to the next result, the asymptotic behavior of the upper tail deviation inequality is consistent with this Gaussian tail.

\begin{thm}
{Let $c>0$ be a constant satisfying \eqref{edgeIneq}.} Then there exists a constant $C>0$ and a sequence $(x_n)_{n\in\N}$ with $x_n \to \infty$ as $n\to\infty$ such that for any $n\in\N$,
\begin{align*}
\exp\left(-\frac{(((\V N_n)^{1/2}r+\E N_n)^{1/4} - (\E N_n)^{1/4})^2}{2c}\right) \leq \exp(-C r^2) \ \ \text{for all} \ \ r\in[0,x_n].
\end{align*}
\end{thm}

\begin{proof}
It was pointed out in \cite{RS_2013} that there are constants $\alpha,\beta > 0$ such that
\begin{align*}
\V N_n \sim \alpha n^3,\\
\E N_n \sim \beta n^2.
\end{align*}
{Let $A = \sqrt{2Cc}$}. Then the desired inequality is equivalent to
\begin{align*}
(\V N_n)^{1/2} r &\geq \left(A r + (\E N_n)^{1/4}\right)^4 - \E N_n\\
&= A^4 r^4 + 4 A^3 r^3 (\E N_n)^{1/4} + 6  A^2 r^2(\E N_n)^{2/4}
+ 4  A r (\E N_n)^{3/4}.
\end{align*}
This holds if and only if
\begin{align*}
(\V N_n)^{1/2}(\E N_n)^{-3/4} - 4 A \geq A^4r^3 (\E N_n)^{-3/4} + 4A^3 r^2 (\E N_n)^{-2/4} + 6 A^2 r (\E N_n)^{-1/4}.
\end{align*}
Now, choose $C>0$ such that {$4A<\alpha^{1/2}\beta^{-3/4}$}. Consider the equality corresponding to the inequality in the above display. For any $n\in\N$ let $x_n$ be the (unique) positive solution of this equality in case such a solution exists, and let $x_n = 0$ otherwise. Then, since the right-hand side of the last display is increasing in $r$, the desired inequality holds for all $r\in[0,x_n]$. Moreover, the left hand side converges to {$\alpha^{1/2}\beta^{-3/4} - 4 A>0$} while $(\E N_n)^{-3/4}, (\E N_n)^{-2/4}, (\E N_n)^{-1/4} \to 0$, as $n\to\infty$. From this it follows that $x_n\to\infty$ as $n\to\infty$.
\end{proof}

\section{Another look at U-Statistics of order two}

In this section, we develop a different {approach } for obtaining deviation inequalities for the upper tail of U-statistics of order 2, that is partially inspired by the results from \cite{HRB, R_2003}. {Throughout this section, we let the assumptions of Section \ref{AppUstat} prevail; in particular, the intensity $\mu$ of $\eta$ is {a non-atomic positive measure on $(\X, \mathcal{X})$}. We begin by generalizing \cite[Theorem 3]{R_2003} to Poisson processes with possibly non-finite intensity measure:

\begin{thm} \label{supThm}
Consider a countable family $\{f_j\}_{j\in J}$ of functions $\X\to[0,1]$ and let
\begin{align*}
G = \sup_{j\in J} \sum_{x\in\eta} f_j(x).
\end{align*}
Assume that $\E G < \infty$. Then for any $\lambda>0$ we have
\begin{align*}
\log \E[\exp(\lambda(G-\E G))] \leq \phi(\lambda)\E G,
\end{align*}
where $\phi(\lambda) = e^\lambda - \lambda - 1$.
\end{thm}

\begin{proof}
First note that by monotone convergence, we can assume without loss of generality that $|J|<\infty$. For each $n\in\N$ let $G_n = \min(G,n)$. Then $\E (e^{\lambda G_n}) < \infty$, hence Proposition \ref{EntIneq} with $I = \emptyset$ gives
\begin{align}\label{supEntIneq}
\Ent(e^{\lambda G_n}) \leq \E \left(e^{\lambda G_n} \sum_{x\in\eta} \phi(-\lambda D_x G_n(\eta-\delta_x))\right).
\end{align}
Consider some realization of $\eta$. Since we assumed $|J|<\infty$, it follows that for some $j^* \in J$ we have
\begin{align*}
G(\eta) = \sum_{x\in\eta} f_{j^*}(x).
\end{align*}
Now, for any $x\in\eta$ we have
\begin{align*}
0\leq D_x G(\eta - \delta_x) \leq f_{j^*}(x) \leq 1.
\end{align*}
Moreover, if $G(\eta-\delta_x) \geq n$, then $D_xG_n(\eta-\delta_x) = 0$ and if $G(\eta - \delta_x) < n$, then
\begin{align*}
D_xG_n(\eta-\delta_x) = D_xG(\eta-\delta_x) - \max(0,G(\eta)-n).
\end{align*}
From this we obtain
\begin{align*}
\sum_{x\in\eta} D_xG_n(\eta-\delta_x) &\leq \left(\sum_{x\in\eta} f_{j^*}(x)\right) - \max(0,G(\eta) - n)\\ &= G(\eta) - \max(0, G(\eta)-n) = G_n(\eta).
\end{align*}
Since $\phi(-\lambda z) \leq \phi(-\lambda) z$ for $\lambda>0$ and $0\leq z \leq 1$, it follows from the above considerations that
\begin{align*}
\Ent(e^{\lambda G_n}) \leq \phi(-\lambda) \E \left(e^{\lambda G_n}  \sum_{x\in\eta} D_x G_n(\eta-\delta_x)\right)  \leq \phi(-\lambda) \E(e^{\lambda G_n}  G_n).
\end{align*}
Continuing in the same way as in the proof of \cite[Theorem 10]{M_2000} gives
\begin{align}\label{truncSupIneq}
\log \E [\exp(\lambda(G_n - \E G_n))] \leq \phi(\lambda)\E G_n \leq \phi(\lambda)\E G.
\end{align}
{Now, since $\lambda>0$ and $\E G < \infty$, by monotone convergence we have}
\begin{align*}
\lim_{n\to\infty}\E[\exp(\lambda(G_n - \E G_n))] = \E[\exp(\lambda(G - \E G))].
\end{align*}
Invoking (\ref{truncSupIneq}) yields the result.
\end{proof}
We continue with an analytic lemma that is used in the proof of the upcoming theorem, but which might be of independent interest in similar situations. The proof is inspired by the proof of \cite[Corollary 2.12]{L_1999}.
\begin{lem}\label{analyticLem}
For any $z>0$,
\begin{align*}
\sup_{\lambda>0}\left[\lambda z - e^{\lambda^2} + 1\right] \geq \frac{\sqrt{\log(z+1)} z^ {3/2}}{4\sqrt{z} + 8}.
\end{align*}
\end{lem}
\begin{proof}
We begin with a preparing observation. Let $a>1$ and consider the map $\lambda \mapsto 1 - e^{\lambda^2} + a \lambda^2$. Then this map takes $0$ to $0$ and it has a unique local extremum (a maximum) at $\sqrt{\log(a)}$ on the positive reals. Hence, whenever $1- e^{x^2}+ax^2 \geq 0$ for some $x>0$, we have $1-e^{\lambda^2} \geq -a \lambda^2$ for any $\lambda\in(0,x]$.

Now assume that $0<z\leq 4$. Then we take $\lambda = z/4$. Since $\lambda\leq 1$ and $1-e + 2\geq 0$, the above observation implies
\begin{align*}
1-e^{\lambda^2} \geq -2\lambda^2.
\end{align*}
Hence, we have
\begin{align*}
\lambda z - e^{\lambda^2} + 1 \geq 
\lambda z -2\lambda^2 = \frac{z^2}{8} \geq \frac{\sqrt{\log(z+1)} z^ {3/2}}{4\sqrt{z} + 8}.
\end{align*}
For $4<z\leq 12$ we take $\lambda = z / 8$. Then $\lambda\leq 3/2$ and since $1-e^{9/4} + 9\geq 0$, the initial observation gives
\begin{align*}
1-e^{\lambda^2}\geq -4 \lambda^2.
\end{align*}
Thus,
\begin{align*}
\lambda z - e^{\lambda^2} + 1 \geq 
\lambda z -4\lambda^2 = \frac{z^2}{16} \geq \frac{\sqrt{\log(z+1)} z^ {3/2}}{4\sqrt{z} + 8}.
\end{align*}

It remains to prove that the result holds for $z>12$. Here we can take $\lambda = \sqrt{\log(z)}$ and see that the supremum is lower bounded by
\begin{align*}
z(\sqrt{\log(z)} - 1) + 1.
\end{align*}
Now, since $\sqrt{\log(z)} - \sqrt{\log(z+1)}$ is increasing for $z\geq 12$, we have
\begin{align*}
\sqrt{\log(z)} \geq \sqrt{\log(z+1)} + A,
\end{align*}
where $A=\sqrt{\log(12)} - \sqrt{\log(13)} < 0$. Hence,
\begin{align*}
z(\sqrt{\log(z)} - 1) + 1 \geq z(\sqrt{\log(z+1)} + A - 1) + 1.
\end{align*}
We claim that
\begin{align*}
z(\sqrt{\log(z+1)} + A - 1) + 1 \geq \frac{1}{4}z(\sqrt{\log(z+1)}) \ \ \text{for} \ \ z>12.
\end{align*}
This will imply the result since
\begin{align*}
\frac{1}{4}z(\sqrt{\log(z+1)}) \geq \frac{1}{4}\frac{z(\sqrt{\log(z+1)}) \sqrt{z}}{\sqrt{z} + 2} = \frac{\sqrt{\log(z+1)}z^{3/2}}{4\sqrt{z} + 8}.
\end{align*}
To prove the claim, define
\begin{align*}
h(z) &= z(\sqrt{\log(z+1)} + A - 1) + 1 -\frac{1}{4}z(\sqrt{\log(z+1)})\\
&= \frac{3}{4}z \sqrt{\log(z+1)} + z(A - 1) + 1.
\end{align*}
The derivative of $h$ satisfies
\begin{align*}
h'(z) = \frac{3}{4} \left(\frac{z+ 2 \log(z+1) (z+1)}{2 \sqrt{\log(z+1)} (z+1)}\right) + A - 1 \geq \frac{3}{4} \sqrt{\log(z+1)} + A - 1.
\end{align*}
Since the right hand side is increasing in $z$ and positive for $z = 12$, we have that $h'(z)\geq 0$ for any $z>12$. So $h$ is increasing for $z>12$ and since also $h(12)\geq 0$, it follows that $h(z)\geq 0$ for all $z>12$. This proves the claim and concludes the proof of the result.
\end{proof}

\begin{thm}\label{UStatDevIneq}
{Let $F$ be a $U$-statistic} of order $2$ with kernel $f\geq 0$ such that $\E F<\infty$. Assume that there is a countable family $\{g_j\}_{j\in J}$ of functions $\X\to[0,1]$ and a constant $c>0$ such that
\begin{align*}
G = \sup_{j\in J} \sum_{x\in\eta} g_j(x)
\end{align*}
satisfies almost surely
\begin{align*}
\sup_{y\in\eta} \sum_{x\in\eta\setminus y} f(y,x) \leq c G
\end{align*}
and $\E G < \infty$. Then for any $r>0$ we have
\begin{align*}
\P(F \geq \E F+r) &\leq \exp\left( -\E (G)\ \chi\left(\frac{\sqrt{\E F + r} - \sqrt{\E F}}{\sqrt{4c}\ \E G}\right)\right),
\end{align*}
where
\begin{align*}
\chi(z)= \frac{\sqrt{\log(z+1)} z^ {3/2}}{4\sqrt{z} + 8}.
\end{align*}
\end{thm}

\begin{proof}
The assumptions imply {that almost surely}
\begin{align*}
V^+ = 4 \sum_{y\in\eta}\left(\sum_{x\in\eta\setminus y} f(y,x)\right)^2 \leq 4cG F.
\end{align*}
By Corollary \ref{cor2} and Theorem \ref{supThm} this gives for any $\lambda> 0$,
\begin{align*}
\log \E [\exp(\lambda(\sqrt{F} - \E \sqrt{F}))] &\leq  \inf_{\theta\in(0,2/\lambda)} \frac{\lambda\theta}{2-\lambda\theta} \log \E \left[\exp\left(\frac{4c\lambda}{\theta} G\right)\right]\\
&\leq \inf_{\theta\in(0,2/\lambda)} \frac{\lambda\theta}{2-\lambda\theta} \E G (\exp(4c\lambda/\theta) - 1)\\
&\leq \E G (\exp(4c\lambda^2) - 1).
\end{align*}
Let $r>0$. Then, using the above computation and Markov's inequality, we obtain for any $\lambda>0$,
\begin{align*}
\P(F \geq \E F + r) &\leq \P(e^{\lambda(\sqrt{F} - \E\sqrt{F})} \geq e^{\lambda( \sqrt{\E F + r} - \sqrt{\E F})})\\
&\leq \exp\left(\E G (\exp(4c\lambda^2) - 1) - \lambda( \sqrt{\E F + r} - \sqrt{\E F})\right).
\end{align*}
Hence, writing $z = (\sqrt{\E F + r} - \sqrt{\E F})/(\sqrt{4c}\E G)$ and substituting $\lambda$ by $\lambda / \sqrt{4c}$, we obtain
\begin{align*}
\P(F\geq \E F + r) \leq \exp\left(-\E G \sup_{\lambda>0} \left[\lambda z -\exp(\lambda^2) + 1 \right]\right).
\end{align*}
The result now follows from Lemma \ref{analyticLem}.
\end{proof}

\subsection{Length power functionals}\label{ss:lpf} As an application, we now focus on length power functionals in random geometric graphs. These estimates contain as special case the edge counting statistics that we studied in Section \ref{s:edge}. We will see in particular that one can take advantage of the upper tail estimate stated in Theorem \ref{UStatDevIneq} in order to provide an alternate bound to that appearing in \eqref{e:copa}, which actually displays a {\it strictly faster} rate of decay in $r$.

As we did in Section \ref{s:edge}, we consider a Poisson measure on $\R^d$ with $\sigma$-finite {and non-atomic Borel intensity measure $\mu$}. We let $\rho>0$ be some radius and consider again the disk graph $\mathfrak{G}(\eta)$ associated with $\eta$. For any $\alpha\in[0,1]$ the {\it length power functional} $L^{(\alpha)}$ is the $U$-statistic of order $2$ with kernel 
\begin{align*}
f_\alpha(x,y) = \tfrac 12 \ind\{\lVert x - y \rVert\leq \rho\} \lVert x - y \rVert^\alpha.
\end{align*}
Note that $L^{(0)} = N$, the number of edges in $\mathfrak{G}(\eta)$, and $L^{(1)}$ is just the {(total edge)} length of the graph. One easily sees that
\begin{align*}
\sup_{y\in\eta} \sum_{x\in\eta\setminus y} f_\alpha(x,y) \leq 2^{d-1} \rho^\alpha \sup_{j\in \N^d} \sum_{x\in\eta} \ind\{x \in[0,2\rho]^d + 2\rho j\}.
\end{align*}
It is also straightforward to check that, if $\E N<\infty$, then the expectation of the right-hand side of the above inequality is finite. We stress also that, if $\E N<\infty$, then $L^{(\alpha)}$ is trivially well-behaved for every $\alpha\in [0,1]$; see Section \ref{s:representative}. The following consequence of Theorem \ref{UStatDevIneq} therefore holds.
\begin{cor}\label{c:we} Let the above notation prevail, fix $\alpha\in [0,1]$, assume that $\E N<\infty$, and let
\begin{align*}
G = \sup_{j\in \N^d} \eta([0,2\rho]^d + 2\rho j).
\end{align*}
Then, $\E G<\infty$ and, for all $r>0$,
\begin{align}\label{e:cd}
\P(L^{(\alpha)} \geq \E L^{(\alpha)} + r) \leq \exp\left( -\E (G)\ \chi\left(\frac{\sqrt{\E L^{(\alpha)} + r} - \sqrt{\E L^{(\alpha)}}}{\sqrt{2^{d+1} \rho^\alpha} \ \E G}\right)\right),
\end{align}
where $\chi(z)$ is as in Theorem \ref{UStatDevIneq}.
\end{cor}

\begin{rem} The right-hand side of \eqref{e:cd} has the form $\exp(-I(r))$, where $I(r) \sim b \sqrt{r\log r}$, for some $b>0$, as $r\to \infty$. Such a rate of decay is better than the one we can deduce from \cite[Proposition 5.1]{RST_2013} (which is indeed a translation of the results from \cite{ERS_2015}), that applies to the case where $\mu$ is a multiple of the restriction of the Lebesgue measure to a convex body, and implies an upper bounds of the form $\exp(-I_0(r))$, with $I_0(r) \sim b_0 r^{1/3}$. Our result provides also a rate of decay that is faster than the one appearing in \cite[Theorem 5.5]{RST_2013}, where the bound has the form $\exp(-I_1(r))$, with $I_1(r) \sim b_1 r^{1/2}$. It is remarkable that, in the case $\alpha =0$ and as far as the rate of decay (as $r\to\infty$) is concerned, the estimate \eqref{e:cd} is also strictly better than \eqref{e:copa}, and that this comes at the cost of somewhat more complicated constants. Finally, we observe that the asymptotic relation $I(r) \sim b \sqrt{r\log r}$ is consistent with Proposition \ref{p:ny}.

\end{rem}

\subsection{Length in more general graph models}
In the following, we consider a slightly more general model of random geometric graphs. For this, let $\eta$ be a Poisson point process on $\R^d$ {with $\sigma$-finite and non-atomic Borel intensity measure $\mu$}, and let $\rho:\R^d\to\R_+$ be given by 
\begin{equation}\label{e:rho}
\rho(x) =\left( \frac{1}{ \lVert x \rVert + 1}\right)^{\gamma},
\end{equation} 
for some $\gamma>0$. We define $\mathfrak{H}(\eta)$ as the graph with vertex set $\eta$ and an edge between vertices $x,y\in\eta$ whenever $0<\lVert x-y \rVert \leq \rho(x)+\rho(y)$. Note that the graph $\mathfrak{H}= \mathfrak{H}(\eta)$ is obtained by implementing the following two-step procedure: (a) for every $x\in \eta$, draw the closed ball {$B(x,\rho(x))$}, centered at $x$ and with radius $\rho(x)$, and (b) connect two distinct points $x,y\in \eta$ with an edge, if and only if $$B(y,\rho(y))\cap B(x,\rho(x)) \neq \emptyset.$$ In other words, $\mathfrak{H}$ is the {\it intersection graph} of the balls centered at the points of $\eta$ with (decaying) radii given by $\rho(x), x\in\eta$. We will see that this model allows situations where $\mathfrak{H}$ has almost surely infinitely many edges but still a finite length. Interestingly, even if there are infinitely many edges, the length can have an exponentially decaying upper tail. Before we analyse concentration properties of the length, we present an illustration of how the considered graph might look like in the plane.

\begin{figure}[h]
\includegraphics[scale=0.26]{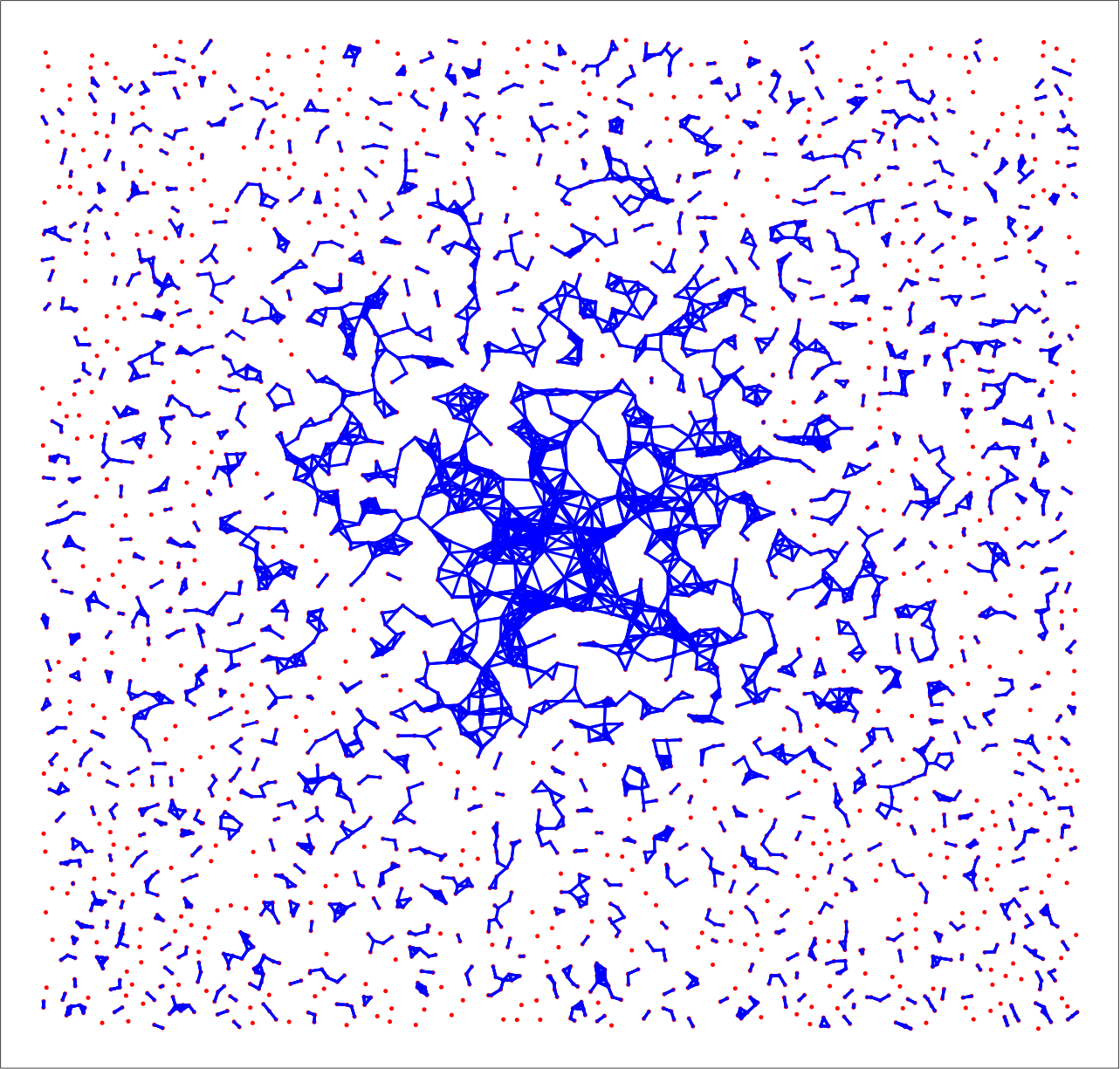}
\caption{A realisation of the intersection graph $\mathfrak{H}(\eta)$}
\end{figure}

Let $L$ be the length of $\mathfrak{H}$. Then $L$ is a U-statistic of order $2$ with kernel 
\begin{align*}
f_L(x,y) = \tfrac 12 \lVert x-y \rVert \ind\{\lVert x-y \rVert \leq \rho(x) + \rho(y)\}.
\end{align*}
{Moreover, if $\mu$ guarantees that almost surely $L<\infty$, then $L$ is even well-behaved. As a consequence of Theorem \ref{UStatDevIneq} we obtain the following result.
\begin{cor} \label{lengthConc}
Let $c = 3^\gamma + 1$ and define
\begin{align}\label{defG}
G = \sup_{x\in\Q^d} \rho(x) \eta(B(x, c\rho(x))).
\end{align}
Assume that $\E G<\infty$ and $\E L < \infty$. Then for any $r\geq 0$,
\begin{align*}
\P(L\geq \E L + r) \leq \exp\left( -\E (G)\ \chi\left(\frac{\sqrt{\E L + r} - \sqrt{\E L}}{\sqrt{2c}\ \E G}\right)\right),
\end{align*}
where $\chi$ is defined as in Theorem \ref{UStatDevIneq}.
\end{cor}

\begin{proof}
Let $x,y\in\R^d$ be such that $\lVert x-y \rVert \leq \rho(x) + \rho(y)$. Then we have $\lvert\lVert x \rVert - \lVert y \rVert \rvert \leq 2$ since $\rho$ is upper bounded by $1$. Hence,
\begin{align}\label{rhoBound}
\frac{\rho(y)}{\rho(x)} = \left(\frac{\lVert x\rVert + 1}{\lVert y\rVert + 1}\right)^\gamma \leq \left(\frac{\lVert y\rVert + 3}{\lVert y\rVert + 1}\right)^\gamma \leq 3^\gamma.
\end{align}
Therefore, $\lVert x-y \rVert \leq (3^\gamma + 1)\rho(x)$. It follows that the local version of $L$ satisfies
\begin{align}\label{lengthObs}
\nonumber L(x,\eta) &= \tfrac 12 \sum_{y\in\eta\setminus x} \lVert x-y\rVert \ind\{\lVert x-y \rVert \leq \rho(x) + \rho(y)\}\\
&\leq \tfrac {3^\gamma + 1}{2} \sum_{y\in\eta\setminus x} \rho(x) \ind\{\lVert x-y \rVert \leq (3^\gamma + 1)\rho(x)\}.
\end{align}
For $x,y\in\R^d$ we define $g_x(y) = \rho(x) \ind\{\lVert x-y \rVert \leq (3^\gamma + 1)\rho(x)\}$. Then the above reasoning gives {that, almost surely,}
\begin{align*}
\sup_{x\in\eta} \sum_{y\in\eta\setminus x} f_L(x,y) \leq \tfrac {3^\gamma + 1}{2} \sup_{x\in\Q^d} \sum_{y\in\eta} g_x(y).
\end{align*}
Let
\begin{align*}
G = \sup_{x\in\Q^d} \sum_{y\in\eta} g_x(y).
\end{align*}
We see that Theorem \ref{UStatDevIneq} applies to $L$ whenever $\E L, \E G<\infty$ and this concludes the proof.
\end{proof}

Next we will prove a sufficient condition for the finiteness of the expectations appearing in Corollary \ref{lengthConc} for the case when the Poisson process $\eta$ is homogeneous, that is, when the intensity measure of $\eta$ has the form $t\times \lambda$, where $\lambda$ is the Lebesgue measure.

\begin{prop}\label{suffCondLG}
Assume that $\eta$ is a homogeneous Poisson point process on $\R^d$ with intensity $t\lambda$, $t>0$. Let $L$ be the length of $\mathfrak{H}$ and define the random variable $G$ as in (\ref{defG}). Then $G$ and $L$ are integrable, provided that
\begin{align}\label{finiteIntCond}
\int_{\R^d}\rho(x)^{d+1} dx < \infty.
\end{align}
\end{prop}

\begin{proof}
First observe, quite similarly as it was done in (\ref{rhoBound}), that for any $x\in\Q^d$ and $\hat{x} \in B(x,(3^\gamma + 1) \rho(x))$,
\begin{align*}
\rho(x) \leq (3^\gamma + 2)^\gamma\rho(\hat{x}).
\end{align*}
Hence, writing $c = 3^\gamma + 1$ and $c' = (3^\gamma + 2)^\gamma$, we have
\begin{align*}
\sum_{y\in\eta} g_x(y) &= \rho(x) \sum_{y\in\eta} \ind\{ \lVert x -y \rVert \leq c\rho(x)\}\\
&\leq c'\rho(\hat{x}) \sum_{y\in\eta} \ind\{ \lVert \hat{x}-y \rVert \leq 2cc'\rho(\hat{x})\}.
\end{align*}
It follows that $\E G<\infty$ if the expectation of the following is finite:
\begin{align} \label{lengthObs2}
\nonumber &\sup_{x\in \eta} \sum_{y\in\eta\setminus x} \rho(x) \ind\{\lVert x-y\rVert \leq 2cc' \rho(x)\}\\
&\leq \sum_{(x,y)\in\eta_{\neq}^2} \rho(x)\ind\{\lVert x-y\rVert \leq 2cc' \rho(x)\}.
\end{align}
Using the Slivniak-Mecke formula \eqref{e:smecke}, we obtain that the expectation of the latter expression equals
\begin{align*}
t^2 \int_{\R^d} \int_{\R^d} \rho(x)\ind\{\lVert x-y\rVert \leq 2cc' \rho(x)\}\ dx\ dy 
= t^2 (2cc')^d \kappa_d \int_{\R^d}\rho(x)^{d+1} dx,
\end{align*}
where $\kappa_d$ denotes the Lebesgue measure of the unit ball in $\R^d$. So we have $\E G<\infty$ provided that (\ref{finiteIntCond}) holds. This condition also guarantees $\E L<\infty$ (to see this, just perform a computation similar to the one above where the estimate (\ref{lengthObs}) is used instead of (\ref{lengthObs2})).
\end{proof}

We initially claimed that $\mathfrak{H}$ can have a.s. infinitely many edges but still $\E L < \infty$. The following result, when {combined} with Proposition \ref{suffCondLG}, substantiates this statement.

\begin{prop} \label{infiniteEdgesProp}
Assume that $\eta$ is a homogeneous Poisson point process in $\R^d$ with intensity $t>0$. Denote by $N$ the number of edges in $\mathfrak{H}$. Then almost surely $N = \infty$, provided that
\begin{align}\label{infEdgesCond}
\int_{\R^d} \rho(x)^d dx = \infty.
\end{align}
\end{prop}

\begin{rem}
The phenomenon described above is remarkable since it allows for situations where the U-statistic
\begin{align*}
L = \sum_{(x,y)\in\eta_{\neq}^2} f_L(x,y)
\end{align*}
is indeed (as opposed to the edge counting statistics) almost surely an infinite series, i.e. we have almost surely $f_L(x,y)>0$ for infinitely many $(x,y) \in\eta_{\neq}^2$. Intuitively, one might expect that in this situation strong concentration properties for $L$ are more difficult to establish. However, as we have seen above, our method works without problems and yields exponential tail bounds for the graph length regardless of whether finitely or infinitely many edges are present. 
\end{rem}

\begin{proof}[Proof of Proposition \ref{infiniteEdgesProp}]
Note that
\begin{align*}
N = \tfrac 12 \sum_{(x,y)\in\eta^2_{\neq}} \ind\{\lVert x-y \rVert \leq \rho(x) + \rho(y)\}.
\end{align*}
For all $x\in\N^d$ define the cube
\begin{align*}
Q_x = [0,\tfrac{1}{\sqrt{d}}\rho(x + \textbf{1})]^d \subset\R^d,
\end{align*}
where $\textbf{1} = (1,\ldots,1)\in\N^d$. Observe that for any $x\in\N^d$, we can place
\begin{align*}
\left\lfloor \frac{\sqrt{d}}{\rho(x + \textbf{1})}\right\rfloor^d =: r(x)
\end{align*}
many disjoint translated copies of $Q_x$ into the cube $x+[0,1]^d$. Denote these copies by $Q_x^1,\ldots Q_x^{r(x)}$. Observe also that for the diameter of each $Q_x^i$ {one has $\diam(Q_x^i) = \rho(x + \textbf{1})$.} Hence, any two distinct vertices $x,y\in \eta$ within the same cube are connected by an edge. Therefore,
\begin{align*}
N\geq \sum_{x\in\N^d} \sum_{i=1}^{r(x)}\binom{\eta(Q_x^i)}{2}.
\end{align*}
Now, the $\eta(Q_x^i)$ are independent Poisson random variables. Thus, by the second Borel-Cantelli lemma, the right hand side in the above display is almost surely non-finite if
\begin{align*}
\sum_{x\in\N^d} \sum_{i=1}^{r(x)} \P(\eta(Q_x^i)\geq 2) = \infty.
\end{align*}
The expectation of $\eta(Q_x^i)$ is given by
\begin{align*}
\lambda_x = \frac{t}{d^{d/2}} \rho(x + \textbf{1})^d.
\end{align*}
Using this we obtain
\begin{align*}
\sum_{x\in\N^d}\sum_{i=1}^{r(x)} \P(\eta(Q_x^i)\geq 2) &=  \sum_{x\in\N^d} r(x)(1 - e^{-\lambda_x} - \lambda_x e^{-\lambda_x}).
\end{align*}
Since $1-e^{-z}-ze^{-z} \geq (1/4) z^2$ for $z\in[0,1]$ and since $\lambda_x\to 0 $ as $\lVert x\rVert\to\infty$, the above series is non-finite if
\begin{align*}
\sum_{x\in\N^d} r(x) \lambda_x^2 = \infty \ \ \ \text{or equivalently} \ \ \ \sum_{x\in\N^d} \rho(x + \textbf{1})^d = \infty.
\end{align*}
It is easy to see that the above is implied by condition (\ref{infEdgesCond}).
\end{proof}

\section{Concentration for the convex distance in Poisson-based models}\label{ss:convex}

The convex distance for product spaces that was introduced by M. Talagrand in \cite{T_1995} has proved to be a very useful tool in the context of concentration inequalities -- see e.g. \cite[Chapter 11]{DP}, \cite[Chapter 6]{steele}, \cite[Chapter 2]{Tao} and the references therein. In the recent paper \cite{R_2013} by M. Reitzner, this notion has been adapted for models based on Poisson point processes with finite intensity measure. For both the product space and the Poisson space version, the method of using the convex distance to establish concentration properties is based on an isoperimetric inequality. {First applications of this method for Poisson-based models are worked out in \cite{RST_2013, LR_2015} where concentration inequalities for Poisson U-statistics are presented.}

The proof of the convex distance isoperimetric inequality in \cite{R_2013} uses an approximation of the Poisson process by binomial processes. The goal in this section is to give an alternative proof for this inequality. Apart from slightly worse constants, we entirely recover Reitzner's result \cite[Theorem 1.1]{R_2013} with the tools developed in the present work. In particular, we only use methods from Poisson process theory, thus answering the question proposed in \cite{R_2013} of whether such a direct proof is possible. Moreover, the assumptions on the space $\X$ for our results are less restrictive than in \cite{R_2013} where only locally compact second countable Hausdorff spaces are considered.

The upcoming presentation is based on \cite{BLM_2009} and \cite{BLM_2003} where the convex distance for product spaces is recovered using the entropy method.

\subsection{Convex distance for Poisson processes}
To introduce the convex distance for Poisson point processes, let $\OX_{\rm fin} \subset \OX$ denote the space of finite integer-valued measures on $\X$ which is equipped with the $\sigma$-algebra $\mathcal{N}_{\rm fin}$ obtained by restricting $\mathcal{N}$ to $\OX_{\rm fin}$. We will write $\xi(x)=\xi(\{x\})$ whenever $\xi\in\OX_{\rm fin}$ and $x\in\X$ in order to simplify notations. For any two measures $\xi,\nu\in\OX_{\rm fin}$, we define the measure $\xi\setminus\nu$ by
\begin{align*}
\xi\setminus\nu = \sum_{x\in\xi}(\xi(x) - \nu(x))_+ \delta_x,
\end{align*}
where $x\in\xi$ indicates that $x$ belongs to the support of $\xi$. The \emph{convex distance} $d_T(\xi, A)$ is now defined for any measurable set $A\in\mathcal{N}_{\rm fin}$ and $\xi\in\OX_{\rm fin}$ by
\[
d_T(\xi,A) = \sup_{\lVert u \rVert_\xi \leq 1}\inf_{\nu\in A}\int_{\X} u\ d(\xi\setminus \nu),
\]
where the supremum ranges over all measurable maps $u:\X \to \R$ such that $\lVert u \rVert_\xi\leq 1$ and $\lVert \cdot \rVert_\xi$ denotes the $2$-norm with respect to the measure $\xi$. It is immediate from the above definition that
\begin{align}\label{e:tuba}
d_T(\xi,A) = \sup_{\lVert u \rVert_\xi \leq 1}\inf_{\nu\in A}\sum_{x\in\xi} u(x)(\xi(x)-\nu(x))_+.
\end{align}
The following result gives an alternative characterization for the convex distance which will be crucial for our proof of the isoperimetric inequality later on.

\begin{prop} \label{ConvDistProp}
Let $A\in\mathcal{N}_{\rm fin}$ and denote by $\mathcal{M}(A)$ the set of probability measures on $A$. Then, for any $\xi\in \OX_{\rm fin}$ we have
\begin{align*}
d_T(\xi,A) &= \max_{\lVert u \rVert_\xi \leq 1}\min_{\zeta \in \mathcal{M}(A)}\sum_{x\in\xi} u(x) \E_{\zeta(\nu)}[(\xi(x) - \nu(x))_+]\\
&= \min_{\zeta \in \mathcal{M}(A)}\max_{\lVert u \rVert_\xi \leq 1}\sum_{x\in\xi} u(x) \E_{\zeta(\nu)}[(\xi(x) - \nu(x))_+],
\end{align*}
where, here and for the rest of the section, we use the shorthand notation
$$
\E_{\zeta(\nu)}[ h(\nu) ] = \int_A h(\nu ) d\zeta(\nu),
$$
for every positive measurable mapping $h : A\to \R_+$.
\end{prop}

\begin{proof}
Here we adapt arguments from the proof of \cite[Proposition 13]{BLM_2003}. We begin by proving that
\begin{align}\label{convDistEquation}
d_T(\xi,A) &= \sup_{\lVert u \rVert_\xi \leq 1}\inf_{\zeta \in \mathcal{M}(A)}\sum_{x\in\xi} u(x) \E_{\zeta(\nu)}[(\xi(x) - \nu(x))_+].
\end{align}
For any $\nu\in A$ consider the probability measure $\zeta_\nu\in \mathcal{M}(A)$ that is concentrated on $\nu$. Then
\[
\sum_{x\in\xi} u(x) \E_{\zeta_\nu(\nu')} [(\xi(x) - \nu'(x))_+] = \sum_{x\in\xi} u(x)(\xi(x) - \nu(x))_+.
\]
Hence, for any $u$ with $\lVert u \rVert_\xi \leq 1$, we have
\[
\inf_{\zeta \in \mathcal{M}(A)}\sum_{x\in\xi} u(x) \E_{\zeta(\nu)} [(\xi(x) - \nu(x))_+] \leq \inf_{\nu\in A}\sum_{x\in\xi} u(x)(\xi(x) - \nu(x))_+.
\]
On the other hand, for all $\zeta\in\mathcal{M}(A)$ we have
\begin{align*}
\inf_{\nu\in A} \sum_{x\in\xi} u(x)(\xi(x) - \nu(x))_+ &\leq \E_{\zeta(\nu)}\sum_{x\in\xi} u(x)(\xi(x) - \nu(x))_+\\
&= \sum_{x\in\xi} u(x) \E_{\zeta(\nu)}[(\xi(x) - \nu(x))_+].
\end{align*}
Thus,
\[
\inf_{\nu\in A} \sum_{x\in\xi} u(x)(\xi(x) - \nu(x))_+ \leq \inf_{\zeta\in\mathcal{M}(A)}\sum_{x\in\xi} u(x) \E_{\zeta(\nu)}[(\xi(x) - \nu(x))_+].
\]
This establishes equation (\ref{convDistEquation}).

We aim at applying Sion's minimax theorem \cite[Corollary 3.3]{S_1958}. To get prepared for this, first note that the supremum in \eqref{convDistEquation} can obviously by performed with respect to those functions $u:\X\to\R$ satisfying $u(x) = 0$ whenever $x\notin\xi$.  Note also that these functions form a finite dimensional real vector space (whose dimension is given by $\#\{x\in\xi\}$) which will be denoted by $U$. So the supremum is actually taken over
\begin{align*}
U_{\leq 1} = \{u \in U : \lVert u \rVert_\xi \leq 1\}
\end{align*}
which is a convex and compact subset of $U$. Denote by $Q$ the finite set of maps $q:\X\to\N_0$ satisfying $q(x)\leq \xi(x)$ for all $x\in\xi$ and $q(x)=0$ whenever $x\notin\xi$. Moreover, define the map $I$ by
\begin{align*}
I:A \to Q, \nu \mapsto (x\mapsto (\xi(x) - \nu(x))_+).
\end{align*}
Then, for any $\zeta\in\mathcal{M}(A)$ and $x\in\xi$ we have
\[
\E_{\zeta(\nu)} [(\xi(x) - \nu(x))_+] = \E_{I\zeta(q)}[q(x)],
\]
where $I\zeta$ denotes the pushforward measure of $\zeta$ with respect to $I$. Now, instead of taking the infimum in $\mathcal{M}(A)$ we can also minimize in the set of pushforward measures $I\mathcal{M}(A)$. The set $I\mathcal{M}(A)$ coincides with the set of probability measures on $I(A)$, denoted by $\mathcal{M}I(A)$. Observe that $\mathcal{M} I(A)$ is a convex and compact subset in the finite dimensional real vector space of all signed measures on $I(A)$ which we denote by $\mathcal{S}I(A)$. Obviously, the map
\[
U\times \mathcal{S}I(A)\to\R, \ (u,\zeta)\mapsto \sum_{x\in\xi} u(x) \E_{\zeta(q)} [q(x)],
\]
is both linear in $u$ and $\zeta$. Hence, it is also upper semicontinuous and quasi-concave in $u$ and lower semicontinuous and quasi-convex in $\zeta$. According to the above considerations, the assumptions of Sion's theorem are satisfied and we obtain
\begin{align*}
&\sup_{\lVert u \rVert_\xi \leq 1}\inf_{\zeta \in \mathcal{M}(A)}\sum_{x\in\xi} u(x) \E_{\zeta(\nu)} [(\xi(x) - \nu(x))_+] = \sup_{u\in U_{\leq 1}}\inf_{\zeta \in \mathcal{M}I(A)} \sum_{x\in\xi} u(x) \E_{\zeta(q)} [q(x)]\\ &= \inf_{\zeta \in \mathcal{M}I(A)} \sup_{u\in U_{\leq 1}} \sum_{x\in\xi} u(x) \E_{\zeta(q)} [q(x)] = \inf_{\zeta \in \mathcal{M}(A)} \sup_{\lVert u \rVert_\xi \leq 1} \sum_{x\in\xi} u(x) \E_{\zeta(\nu)} [(\xi(x) - \nu(x))_+].
\end{align*}
Since both $U_{\leq 1}$ and $\mathcal{M} I(A)$ are compact, the suprema and infima are actually maxima and minima.
\end{proof}

\subsection{Convex distance inequality}
In what follows, we will give the announced new proof of the convex distance inequality for Poisson point processes. The result we aim to prove is the following:
{
\begin{thm}\label{convDistIneq}
Let $\eta$ be a Poisson point process in $\X$ with finite intensity measure $\mu$. Let $A\in\mathcal{N}_{\rm fin}$ be arbitrary. Then
\begin{align*}
\P(\eta\in A) \E( e^{d_T(\eta, A)^2/ 10}) \leq 1.
\end{align*}
In particular, for any $r\geq 0$,
\begin{align}\label{e:r/9}
\P(\eta\in A)\P(d_T(\eta,A)\geq r) \leq e^{-r^2/10}.
\end{align}
\end{thm}
}
Note that in \cite[Theorem 1.1]{R_2013} (under more restrictive assumptions on the space $(\X, \mathcal{X}, \mu)$), an inequality stronger than \eqref{e:r/9} is proved, where the {constant 1/10} is {replaced by 1/4}. To get prepared for the proof of Theorem \ref{convDistIneq}, we first establish the following result. This is interesting in its own right since it particularly states that the variance of the convex distance is bounded by $1$.

\begin{prop}\label{convDistProp}
Let $\eta$ be a Poisson point process on $\X$ with finite intensity measure $\mu$. Then for any $A\in\mathcal{N}_{\rm fin}$ almost surely
\begin{align*}
V^+(d_T(\eta, A)) = \int_\X(D_x d_T(\eta - \delta_x, A))^2 d\eta(x) \leq 1.
\end{align*}
In particular, $\V d_T(\eta, A)\leq 1$.
\end{prop}

\begin{proof}
For this proof we adapt arguments from the proof of \cite[Proposition 13]{BLM_2003}. According to Proposition \ref{ConvDistProp} we can choose a map $\hat{u}:\X\to\R$ with $\lVert \hat{u}\rVert_\xi\leq 1$ and a probability measure $\hat{\zeta}$ on $A$ satisfying
\[
d_T(\xi,A) = \sum_{x\in\xi} \hat{u}(x) \E_{\hat{\zeta}(\nu)} [(\xi(x) - \nu(x))_+].
\]
Then, for any $z\in\xi$, we have
\[
d_T(\xi - \delta_z, A)\geq \min_{\zeta\in\mathcal{M}(A)} \sum_{x\in\xi-\delta_z} \hat{u}(x) \E_{\zeta (\nu)} [((\xi - \delta_z)(x) - \nu(x))_+].
\]
Choose some $\tilde{\zeta}\in\mathcal{M}(A)$ that achieves the minimum in the above right hand side. Then
\[
d_T(\xi,A)\leq \sum_{x\in\xi} \hat{u}(x) \E_{\tilde{\zeta} (\nu)} [(\xi(x) - \nu(x))_+].
\]
It follows that
\begin{align*}
&d_T(\xi,A) - d_T(\xi- \delta_z,A)\\ &\leq \sum_{x\in\xi} \hat{u}(x) \E_{\tilde{\zeta} (\nu)} [(\xi(x) - \nu(x))_+] - \sum_{x\in\xi-\delta_z} \hat{u}(x) \E_{\tilde{\zeta} (\nu)} [((\xi - \delta_z)(x) - \nu(x))_+]\\
&= \hat{u}(z) \E_{\tilde{\zeta} (\nu)} [(\xi(z) - \nu(z))_+ - (\xi(z) - \nu(z) - 1)_+]\\
&= \hat{u}(z) \E_{\tilde{\zeta} (\nu)} [\ind\{\xi(z) > \nu(z)\}] \leq \hat{u}(z).
\end{align*}
This yields
\[
\int_\X (d_T(\xi,A) - d_T(\xi- \delta_z,A))^2 d\xi(z) \leq \int_\X \hat{u}(z)^2 \xi(z) = \lVert \hat{u}\rVert_\xi^2 \leq 1.
\]
Hence, almost surely
\[
\int_\X (D_x d_T(\eta - \delta_x, A))^2 d\eta(x)\leq 1.
\]
Applying the Poincare Inequality for Poisson processes (see e.g. \cite[Remark 1.4]{W_2000}) yields $\V d_T(\eta, A) \leq 1$.
\end{proof}

As a final ingredient for the upcoming proof of the convex distance inequality, we derive the following consequence of the Cauchy-Schwarz Inequality.

\begin{lem} \label{cauchySchwarzApp}
Let $\xi\in\OX_{\rm fin}$ and consider the measure space $(\X, \mathcal{X}, \xi)$. Then for any measurable map $h:\X\to\R$,
\begin{align} \label{cauchySchwarzEq}
\sup_{\lVert u \rVert_\xi \leq 1} \int_\X u(x)h(x) d\xi(x) = \lVert h \rVert_\xi.
\end{align}
\end{lem}
\begin{proof}
Note that $h$ is of course square-integrable with respect to $\xi$. Hence, by the Cauchy-Schwarz Inequality, for any $u$ such that $\lVert u\rVert_\xi\leq 1$,
\begin{align*}
\int_\X u(x)h(x) d\xi(x) \leq \lVert u\rVert_\xi \lVert h\rVert_\xi \leq \lVert h\rVert_\xi.
\end{align*}
We see that the LHS in (\ref{cauchySchwarzEq}) is less or equal to the RHS. Moreover, we can take $u = h / \lVert h \rVert_\xi$ to conclude that the RHS is less or equal to the LHS.
\end{proof}

\begin{proof}[Proof of Theorem \ref{convDistIneq}]
Here we adapt arguments from the proofs of \cite[Lemma 1 and Corollary 1]{BLM_2009}. We will prove below that
\begin{align}
\label{cond1}&0\leq D_x (d_T(\xi,A)^2) \leq 2 \ \ \text{for any} \ \ (x,\xi)\in\X\times\OX_{\rm fin}\\
\label{cond2}&\text{and almost surely} \ \ V^+(d_T(\eta,A)^2) \leq 4 d_T(\eta,A)^2.
\end{align}
Hence, if follows from Theorem \ref{arbSignThm} and Theorem \ref{vPlusLowerThm} {(where the latter result is applied to the Poisson functional $\tfrac 12 d_T(\eta,A)^2$)} that
\begin{enumerate}
\item For any $\lambda\in(0,1/2)$,
\begin{align*}
\log \E(\exp(\lambda(d_T(\eta,A)^2 - \E d_T(\eta,A)^2))) \leq \frac{2\lambda^2 \E d_T(\eta,A)^2}{1 - 2\lambda},
\end{align*}
\item For any $r\geq 0$,
\begin{align*}
\P(d_T(\eta,A)^2\leq \E d_T(\eta,A)^2 - r) \leq \exp\left(-\frac{r^2}{8\E d_T(\eta,A)^2}\right).
\end{align*}
\end{enumerate}
{Taking $\lambda = 1/10$} we obtain from (i) that
\begin{align*}
\E e^{d_T(\eta,A)^2/{10}} \leq \exp\left(\frac{\E d_T(\eta,A)^2}{8}\right).
\end{align*}
Moreover, since $\eta\in A$ implies $d_T(\eta,A) = 0$, it follows from (ii) with $r = \E d_T(\eta,A)^2$ that
\begin{align*}
\P(\eta\in A) \leq \P\left(d_T(\eta, A)^2 \leq \E d_T(\eta,A)^2 - \E d_T(\eta,A)^2\right) \leq \exp\left(-\frac{\E d_T(\eta,A)^2}{8}\right).
\end{align*}
So, the result follows once we have proven (\ref{cond1}) and (\ref{cond2}). To prove (\ref{cond2}), first observe that $d_T(\cdot,A)$ is a non-decreasing functional. Using this and Proposition \ref{convDistProp} we compute
\begin{align*}
V^+(d_T(\eta,A)^2) &= \int_\X (d_T(\eta,A)^2 - d_T(\eta-\delta_x,A)^2)^2 \ d\eta(x)\\
&= \int_\X (d_T(\eta,A) - d_T(\eta-\delta_x,A))^2 (d_T(\eta,A) + d_T(\eta-\delta_x,A))^2 \ d\eta(x)\\
&\leq \int_\X (d_T(\eta,A) - d_T(\eta-\delta_x,A))^2 4d_T(\eta,A)^2  \ d\eta(x)\\ 
&= 4d_T(\eta,A)^2 \int_\X (D_x d_T(\eta-\delta_x,A))^2   \ d\eta(x) \leq 4d_T(\eta,A)^2.
\end{align*}
It remains to prove (\ref{cond1}). For this, let $(z,\xi)\in\X\times \OX_{\rm fin}$. Then, according to Proposition \ref{ConvDistProp}, we can write
\begin{align*}
d_T(\xi,A) &= \max_{\lVert u \rVert_{\xi} \leq 1} \sum_{x\in\xi} u(x) \E_{\hat{\zeta}(\nu)} [(\xi(x)-\nu(x))_+]\\
&=\max_{\lVert u \rVert_{\xi} \leq 1} \int_\X u(x) \E_{\hat{\zeta}(\nu)} \left[\left(1-\frac{\nu(x)}{\xi(x)}\right)_+\right] \ d\xi(x)
\end{align*}
for some probability measure $\hat{\zeta}$ on $A$. By virtue of Lemma \ref{cauchySchwarzApp}, the latter expression equals
\begin{align*}
\sqrt{\int_\X \left(\E_{\hat{\zeta}(\nu)} \left[\left(1-\frac{\nu(x)}{\xi(x)}\right)_+\right]\right)^2 d\xi(x)}.
\end{align*}
Invoking Proposition \ref{ConvDistProp} and again Lemma \ref{cauchySchwarzApp}, we also obtain
\begin{align*}
d_T({\xi + \delta_z},A) &\leq \max_{\lVert u \rVert_{\xi + \delta_z} \leq 1} \sum_{x\in\xi+\delta_z} u(x) \E_{\hat{\zeta}(\nu)} [((\xi + \delta_z)(x)-\nu(x))_+]\\
&= \sqrt{\int_\X \left(\E_{\hat{\zeta}(\nu)} \left[\left(1-\frac{\nu(x)}{(\xi + \delta_z)(x)}\right)_+\right]\right)^2 d(\xi + \delta_z)(x)}.
\end{align*}
From this it follows that
\begin{align*}
D_zd_T^2 \leq \left(\E_{\hat{\zeta}(\nu)} \left[\left(1-\frac{\nu(z)}{\xi(z)+1}\right)_+\right]\right)^2 (\xi(z)+1) - \left(\E_{\hat{\zeta}(\nu)} \left[\left(1-\frac{\nu(z)}{\xi(z)}\right)_+\right]\right)^2 \xi(z)
\end{align*}
where the subtrahend vanishes whenever $\xi(z)=0$. {Clearly, if $\xi(z)=0$, then the RHS in the above display is less or equal to $1$. So assume that $\xi(z)>0$.} Then, using the abbreviations $G(\nu,z) = (\xi(z)-\nu(z)+1)_+$ and $G'(\nu,z) = (\xi(z)-\nu(z))_+$, one observes that the RHS in the last display can be upper bounded by
\begin{align*}
\frac{(\E_{\hat{\zeta}(\nu)} G(\nu,z))^2 - (\E_{\hat{\zeta}(\nu)} G'(\nu,z))^2}{\xi(z)+1} &= \frac{\E_{\hat{\zeta}(\nu)}  [G(\nu,z) - G'(\nu,z)]\ \E_{\hat{\zeta}(\nu)}[G(\nu,z) + G'(\nu,z)]}{\xi(z)+1}\\
&\leq \frac{\E_{\hat{\zeta}(\nu)}[G(\nu,z) + G'(\nu,z)]}{\xi(z)+1}\\
&=\E_{\hat{\zeta}(\nu)}\left[ \frac{(\xi(z) - \nu(z) + 1)_+ + (\xi(z) - \nu(z))_+}{\xi(z)+1}\right]\\
&\leq 2.
\end{align*}
It follows that $D_z(d_T(\xi,A)^2) \leq 2$.

The functional $d_T(\cdot, A)$ is non-decreasing and non-negative, thus $D_z(d_T(\xi,A)^2)\geq 0$. This concludes the proof.
\end{proof}

\renewcommand{\bibname}{References}
\defbibheading{bibliography}[\refname]{\section*{#1}}
\printbibliography

\end{document}